\documentclass[reqno,12pt]{amsart}
\usepackage{amsfonts}
\usepackage{graphicx}
\usepackage{amsfonts,amsmath, amssymb}
\usepackage{marginnote}
\usepackage{color}
\usepackage{cite}
\usepackage{bm}

\numberwithin{equation}{section}

 \newtheorem{theorem}{Theorem}[section]
 
 \newtheorem{lemma}[theorem]{Lemma}
 \newtheorem{proposition}[theorem]{Proposition}
 \theoremstyle{assumption}
   \newtheorem{assumption}[theorem]{Assumption}
\theoremstyle{definition}
\newtheorem{definition}[theorem]{Definition}
 \theoremstyle{remark}
\newtheorem{remark}[theorem]{Remark}

 \newcommand{\eps}{\varepsilon}

\newcommand{\norm}[1]{\Vert#1\Vert}
 \newcommand{\abs}[1]{\left\vert#1\right\vert}

 \newcommand{\inner}[1]{\left(#1\right)}
\newcommand{\comi}[1]{\left<#1\right>}

 \newcommand{\normm}[1]{{ \vert\kern-0.25ex \vert\kern-0.25ex \vert #1 
     \vert\kern-0.25ex \vert\kern-0.25ex \vert}}

\makeatletter

\def\@startsection#1#2#3#4#5#6{%
 \if@noskipsec \leavevmode \fi
 \par \@tempskipa #4\relax
 \@afterindentfalse
 \ifdim \@tempskipa <\z@ \@tempskipa -\@tempskipa \@afterindentfalse\fi
 \if@nobreak \everypar{}\else
     \addpenalty\@secpenalty\addvspace\@tempskipa\fi
 \@ifstar{\@dblarg{\@sect{#1}{\@m}{#3}{#4}{#5}{#6}}}%
         {\@dblarg{\@sect{#1}{#2}{#3}{#4}{#5}{#6}}}%
}

\def\@settitle{%
  \bgroup
  \centering
  \vglue1cm
  \fontsize{12}{15}\fontseries{b}\selectfont
  \uppercasenonmath\@title
  \@title
  \vskip20pt plus 6pt minus 8pt
  \egroup
}

\def\@setauthors{%
  \begingroup
  \trivlist
  \centering \bfseries
 \normalsize\@topsep30\p@\relax
  \advance\@topsep by -\baselineskip
  \item\relax
  \andify\authors
 {\rmfamily\authors}%
  \endtrivlist
  \endgroup
}

\def\@setaddresses{\par
  \nobreak \begingroup
\normalsize
  \def\author##1{\nobreak\addvspace\bigskipamount}%
  \def\\{\unskip, \ignorespaces}%
  \interlinepenalty\@M
  \def\address##1##2{\begingroup
    \par\addvspace\bigskipamount\noindent
    \@ifnotempty{##1}{(\ignorespaces##1\unskip) }%
    {\ignorespaces##2}\par\endgroup}%
  \def\curraddr##1##2{\begingroup
    \@ifnotempty{##2}{\nobreak\indent{\itshape Current address}%
      \@ifnotempty{##1}{, \ignorespaces##1\unskip}\/:\space
      ##2\par}\endgroup}%
  \def\email##1##2{\begingroup
    \@ifnotempty{##2}{\nobreak\noindent{\itshape E-mail address}%
      \@ifnotempty{##1}{, \ignorespaces##1\unskip}\/: 
       ##2\par}\endgroup}%
   \def\urladdr##1##2{\begingroup
    \@ifnotempty{##2}{\nobreak\indent{\itshape URL}%
      \@ifnotempty{##1}{, \ignorespaces##1\unskip}\/:\space
      \ttfamily##2\par}\endgroup}%
  \addresses
  \endgroup
}

 \renewcommand\section{\@startsection{section}{1}{\z@}%
{27pt plus 6pt minus 8pt}{14pt plus 6pt minus 8pt}
{\center\normalfont\large\bfseries}}
\renewcommand\subsection{\@startsection{subsection}{2}{\z@}%
{27pt plus 6pt minus 8pt}{14pt plus 6pt minus 8pt}
{ \center
\normalfont\bfseries}}

\def\subsubsection{\@startsection{subsubsection}{3}%
  \z@{.5\linespacing\@plus.7\linespacing}{-.5em}%
  {\normalfont\itshape}}

 \def\@mainsize{10}\def\@ptsize{0}%
  \def\@typesizes{%
    \or{5}{6}\or{6}{7}\or{7}{8}\or{8}{9}\or{9}{11}%
    \or{10}{13}
    \or{\@xipt}{13}\or{\@xiipt}{14}\or{\@xivpt}{17}%
    \or{\@xviipt}{20}\or{\@xxpt}{24}}%
  \normalsize \linespacing=\baselineskip

\makeatother

\begin{document}

\title[Well-posedness  of MHD Boundary layer]{Well-posedness of the MHD boundary layer system in Gevrey function space without   Structural Assumption}

\author[W.-X. Li and T. Yang]{ Wei-Xi Li \and Tong Yang}

\date{}

\address[W.-X. Li]{ School of Mathematics and Statistics, Wuhan University, Wuhan 430072, China \& Hubei Key Laboratory of Computational Science, Wuhan University, Wuhan 430072, China
  }

\email{
wei-xi.li@whu.edu.cn}

\address[T.Yang]{
Department of Mathematics, City University of Hong Kong, Hong Kong
  }

\email{
matyang@cityu.edu.hk}

\begin{abstract}
 We establish the well-posedness of  the MHD boundary layer system in Gevrey function space  without any structural assumption.  Compared   to the classical Prandtl equation,    the loss of  tangential derivative comes
  from both the velocity and magnetic fields that are coupled with each other.  By observing 
   a new type of  cancellation mechanism in the system for  overcoming the loss derivative degeneracy, we show
   that the MHD boundary layer system is well-posed with Gevrey index up to  $3/2$ in both two and three dimensional
   spaces.      
\end{abstract}


 \maketitle


\section{Introduction}

Magnetohydrodynamic (MHD)  is concerned with the motion of conducting fluid under the influence of the self-induced
magentic field. In the incompressible framework, the governing equations are
\begin{equation}\label{mhdsys}
\left\{
	\begin{aligned}
		&\partial_t \bm  u +(\bm  u \cdot\nabla) \bm u-(\bm B \cdot\nabla) \bm B +\nabla   P -\frac{1}{\mathrm {Re}} \Delta \bm u  =0,\\
		&\partial_t \bm  B  -\nabla\times(\bm u\times\bm B) -\frac{1}{\mathrm{Rm}} \Delta \bm B=0,\\
		&\nabla\cdot \bm u=\nabla\cdot \bm B=0,\\
		&{\bm u}|_{t=0}={\bm u_0},\quad {\bm B}|_{t=0}={\bm B_0},
	\end{aligned}
\right.
\end{equation}
where  Re and Rm stand for the hydrodynamic and magnetic Reynolds numbers, respectively.     The MHD system is well-explored when the fluid region is the whole space, seeing for instance the survey paper \cite{lin} and the references therein.      Here we assume that the fluid  is  in the half-space 
  $\mathbb R_+^d=\big\{(x_1,\cdots,x_d)\in\mathbb R^d \ | \ x_d>0\big\}$ with   $d=2$ or $d=3$,
and  the system \eqref{mhdsys} is equipped with  the no-slip boundary condition
    on the velocity  field and  perfectly conducting boundary condition   on the magnetic field, that is,
 \begin{eqnarray*}
 	\bm u|_{x_d=0}=\bm 0, \quad \bm (\partial_{x_d}\bm B_h, \bm B_{x_d}) |_{x_d=0}=\bm 0, 
 \end{eqnarray*}
 where $\bm B_h, \bm B_{x_d}$ represent the tangential and normal components of $\bm B$, respectively. 
In this work we will investigate the  well-posedness  of the MHD  boundary layer system derived
 from the   high Reynolds numbers limit of the MHD system \eqref{mhdsys}.   More precisely, when the hydrodynamic and magnetic Reynolds numbers  are of  the same order,  i.e.,  1/Re $=\nu\eps$ and  1/Rm $=\mu\eps$ for $\eps \ll 1$,   the following MHD boundary layer system  was derived in \cite{MR3882222} (cf. the work \cite{MR3657241} for the  derivation with the  insulating boundary condition     on magnetic field):
 \begin{equation}\label{dimmhd}
	\left\{
	\begin{aligned}
		& \big(\partial_t+ \vec{u} \cdot \nabla    -\nu\partial_z^2\big) u_h-(\vec{f} \cdot \nabla)f_h+\nabla_h p =0,\\
		& \partial_t \vec f -\nabla\times (\vec u \times \vec f\,)      -\mu\partial_z^2\vec  f =0, \\
		&\textrm{div}\ \vec u=\textrm{div}\ \vec f=0,\\
		& \vec u |_{z=0}=(\partial_zf_h, f_z)|_{z=0}=\bm 0, \quad (u_h, f_h)|_{z\rightarrow +\infty}=(\bm U,  \bm F),\\
		&  u_h|_{t=0}=  {u}_{h,0},\quad  f_h|_{t=0}=f_{h,0},
	\end{aligned}
	\right.
\end{equation}
where   $x_h\in\mathbb R^{d-1}$ is the tangential component of $(x_h, z)\in\mathbb R_+^d$ and we use the notation $\nabla=\inner{\nabla_h, \partial_z} $ with $\nabla_h=\partial_{x_h}=(\partial_{x_1},\cdots,\partial_{x_{d-1}})$,  and denote  by $\vec u=(u_h,u_z)$ and $\vec f=(f_h, f_z)$   the velocity and magnetic fields respectively,  with the tangential components $u_h, f_h$ and the normal components $u_z, f_z$.       Here    $p$, $\bm U $ and $  \bm F$ are given functions in $(t,x_h)$ variables satisfying  the Bernoulli's law
\begin{eqnarray*}
\left\{
\begin{aligned}
& \partial_t  \bm U + (\bm U\cdot\nabla_h ) \bm U -(\bm F\cdot\nabla_h ) \bm F +\nabla_h p=0,\\
& \partial_t \bm F  + (\bm U\cdot\nabla_h) \bm F -(\bm F\cdot\nabla_h ) \bm U =0.
\end{aligned}
\right.
\end{eqnarray*}
In view of  the  divergence free and boundary  conditions we can write the normal components $u_z$ and $f_z$ as
\begin{eqnarray*}
	u_z(t,x_h, z)=-\int_0^z \nabla_h\cdot u_h(t,x_h,\tilde z) d\tilde z, \quad f_z(t,x_h, z)=-\int_0^z \nabla_h\cdot f_h(t,x_h,\tilde z) d\tilde z.
\end{eqnarray*} 
Thus the MHD boundary layer   system \eqref{dimmhd} is a degenerate  system with the loss of tangential derivatives
  in $f_z$ and $u_z$ as non-local terms.  Note  the equation  for  $f_z$  in \eqref{dimmhd} is just an immediate consequence of those for $f_h$, in view of the representation of $f_z$ given above.     The degeneracy coupled with the non-local property is the main difficulty in studying the well-posedness of this system. 

In the absence of magnetic field,
the MHD   system is  reduced to  the classical incompressible  Navier-Stokes equations,  and the
corresponding  boundary layer system \eqref{dimmhd}   is  the classical Prandtl equation  derived by Prandtl in 1904.
     The mathematical study on   the Prandtl boundary layer has a  long history, and  there have been extensive works concerning its well/ill-posedness theories.   So far  the two-dimensional  (2D) Prandtl equation  is  well-explored in various function spaces,  see e.g. \cite{MR3327535, MR3795028,  MR3925144,  MR2601044, MR1476316, MR2049030,MR2849481,  MR3461362, MR3284569, MR2975371,  MR3590519, MR3493958,  MR2020656, MR3710703,MR3464051} and the references therein. 
  Among these works we can see that there are basically two main settings based on  whether or not  the structural assumption is imposed. One refers to  Oleinik's monotonicity  condition and  another one to the analytic or Gevrey class.  Under Oleinik's monotonicity,  the well-posedness in function space with finite order of regularity
  was firstly achieved by Oleinik (see e.g. \cite{MR1697762})  by using the Crocco transformation, and  was 
  recently proved by two research groups \cite{MR3327535,MR3385340} independently with  new understanding on
   cancellation mechanism through 
  energy method. Hence, the loss of one order tangential  derivative can be overcome by using either Crocco transformation or cancellation mechanism with the monotonicity condition.
  
    Without any structural assumption,  it is natural to introduce  the analytic function space to overcome  the loss of one order derivative by shrinking the  radius of analyticity in time, cf.  \cite{MR1617542}  and the later improvement in  \cite{MR2975371,MR2049030} that hold in both 2D and 3D.  
  Recently some new idea of cancellation was observed in \cite{MR3925144}  to establish the well-posedness in Gevrey function space with index up to $2$ rather than in analytic setting for the 2D Prandtl equation.  
   Compared to the 2D case, much less is known for 3D Prandtl equation outside the
   analytic framework. Here,    we refer to \cite{MR3600083}  for the existence of classical solutions based on some structural assumption such that the secondary flow does not appear,
    and  the work \cite{2018arXiv180404285L} about the wild solution to this system. Recently, 
  the well-posedness in Gevrey space with the Gevrey index up to $2$ was obtained  in \cite{lmy} for 3D Prandtl equations without any structural assumption,  inspired by   the work   \cite{MR3925144} for 2D.  
    Note that  the Gevrey index $2$ in \cite{MR3925144,lmy}  is optimal  in view of the ill-posedness theory  in \cite{MR2601044, MR3458159}.   
    
    Finally,  let us also mention the work \cite{MR2020656} on the global existence of weak solutions under an additional favorable pressure condition and the work \cite{2019arXiv191103690P} on the existence of global solutions in analytic function space for small initial data.  All these results are in fact related to the high Reynolds number limit for the  purely hydrodynamic flow with  physical boundary conditions, and  to show that    the Navier-Stokes equations can
    be approximaged  by the Euler equation away from boundary and  by the Prandtl equation near the boundary.  The  mathematically rigorous  justification of  the limit was obtained by \cite{MR3614755, MR1617538} in the analytic 
    function space without any structural assumption. And there is a significant improvement to the Gevrey setting  in \cite{MR3855356, 2020arXiv200505022G} with    some kind of concave condition  on the Prandtl boundary layer profile.  If the  initial vorticity  is supported away from the boundary then the limit in  $L^\infty$ norm was established in \cite{MR3207194} and  \cite{MR3774877} respectively for 2D and 3D cases.  The aforementioned works on the inviscid limit  are concerned with  the time dependent problem. On the other hand,  for steady flow  we refer to \cite{MR3961300, 2018arXiv180505891G,2019arXiv190308086G,  MR3634071,MR4081398,MR4060005} and references therein for the study of  the inviscid limit in Sobolev or $L^\infty$ setting.   

Back to MHD system,  we  have new difficulty  caused by the additional  loss of  
tangential derivative in the magnetic field.   With some structural condition,  the stabilizing effect of   magnetic field on the boundary layer
has been observed,  see e.g. \cite{MR3657241, MR4102162, MR3975147, MR3882222}. Precisely, under the
assumption on the non-degenerate tangential magnetic field,   the well-posedness of MHD boundary layer in Sobolev space together with the justification of the high Reynolds numbers limit were
 obtained in \cite{MR3882222, MR3975147} without  Oleinik's monotonicity condition on
the velocity field. 
 On the other hand, the magnetic field may act as a destabilizing factor and  lead to  the boundary layer separation, cf. \cite{Nu} .     
Inspired by the well-posedness theory established in \cite{MR3925144,lmy} for the Prandtl equation,   
this paper aims to investigate the well-posedness of the MHD boundary layer without any structural assumption
in the Gevrey function space.   For this,  we need to  explore  other  intrinsic cancellation 
mechanism to overcome the additional  loss of tangential derivative in the magnetic field coupled with
the velocity field.

 To simply the argument we will assume without loss of generality   that  $(\bm U, \bm F)\equiv \bm 0$ in  the system \eqref{dimmhd} because the result  holds in the general case if we use some kind of the non-trivial weighted
 functions similar to those used in \cite{MR3925144} for Prandtl equation.   Hence, we consider
  \begin{equation}\label{3dimmhd}
	\left\{
	\begin{aligned}
		& \big(\partial_t+ \vec{u} \cdot \nabla    -\nu\partial_z^2\big) u_h-(\vec{f} \cdot \nabla)f_h =0,\\
	& \partial_t \vec f -\nabla\times (\vec u \times \vec f\,)      -\mu\partial_z^2\vec  f =0, \\
			&\textrm{div}\ \vec u=\textrm{div}\ \vec f=0,\\
		& \vec u |_{z=0}=(\partial_z f_h, f_z)|_{z=0}=\bm 0, \quad (u_h, f_h)|_{z\rightarrow +\infty}=\bm 0,\\
		& u_h|_{t=0}=  {u}_{h,0},\quad  {f}_h|_{t=0}= {f}_{h,0}.
	\end{aligned}
	\right.
\end{equation}

For clear presentation, let us first introduce the Gevrey  function spaces used in this paper.

  \begin{definition} 
\label{defgev}  Let $\ell\geq 1 $ be a given number.   With a given integer $N\geq 0$ and a pair $(\rho,\sigma)$, $\rho>0$ and $\sigma\geq 1, $  a Banach space $X_{\rho,\sigma,N}$   consists of all  smooth  vector-valued functions
 $\bm A=\bm A(x_h,z)$ with $(x_h,z)\in\mathbb R_+^{d}$  such that the Gevrey norm  $\norm{\bm A}_{\rho,\sigma,N}<+\infty,$  where    $\norm{\cdot}_{\rho,\sigma,N}$ is defined below.   Denote $\partial_{x_h}^\alpha =\partial_{x_1}^{\alpha_1}\cdots\partial_{x_{d-1}}^{\alpha_{d-1}}$ and define
\begin{eqnarray*}	
 \norm{\bm A}_{\rho,\sigma,N}= \sup_{\stackrel{0\leq j\leq N}{\abs\alpha+j\geq 7}} \frac{\rho^{\abs\alpha+j- 7}}{[\inner{\abs\alpha+j- 7}!]^{\sigma}}  \big\|\comi z^{\ell+j} \partial_{x_h}^\alpha \partial_z^j  \bm A \big\|_{L^2(\mathbb R_+^d)}
	 +\sup_{\stackrel{0\leq j\leq N}{\abs\alpha+j\leq 6}}  \big\|\comi z^{\ell+j} \partial_{x_h}^\alpha \partial_z^j  \bm A \big\|_{L^2(\mathbb R_+^d)} ,
\end{eqnarray*}
where  $\comi z=(1+\abs z^2)^{1/2}$ and 
\begin{eqnarray*}
	\norm{\bm A}_{L^2(\mathbb R_+^d)}\stackrel{\rm def}{=}\Big(\sum_{1\leq j\leq k}\norm{A_j}_{L^2(\mathbb R_+^d)}^2\Big)^{1/2}
\end{eqnarray*}
  for $\bm A=(A_1,\cdots, A_k)$.  Here,  $\sigma$ is the Gevrey index. 
\end{definition}

And the main theorem in this paper can be stated as follows.

\begin{theorem} Let the dimension $d=2$ or $3$.   
	Suppose the initial data $(u_{h,0}, f_{h,0}) $ in the system \eqref{3dimmhd} belong to $X_{2\rho_0, \sigma, 8}$  for some  $1< \sigma\leq 3/2$ and some $0<\rho_0 \leq 1$,  compatible with  the boundary condition.  
	 Then the system \eqref{3dimmhd} admits a unique solution $(u_h,f_h) \in L^\infty\big([0,T];~X_{\rho,\sigma,4}\big)$ for some $T>0$ and some $0<\rho<2\rho_0.$ 
\end{theorem}

 Note that  the instability result in  \cite{MR3864769} suggests $\sigma=2$ may be the optimal Gevrey index for the well-posedness theory of the  MHD boundary layer without structural assumption similar to the classical
 Prandt equation.  Hence, it remains 
 an interesting problem to establish a well-posedness theory in Gevrey function space with optimal index.   

The rest of the paper is organized as follows.      For clear presentation, we will  prove  in Section  \ref{secapri} the well-posedness of the  2D MHD boundary layer system.   The discussion on  3D MHD will be given in Section \ref{sec3d} by pointing out the difference.

\section{2D MHD boundary layer}\label{secapri}

For the 2D MHD boundary layer system,  we will use $(u, w) $ and $(f,h)$ to denote the velocity and magnetic fields respectively, and denote by $(x,z)\in\mathbb R_+^2$ the spatial variable.   Then  the MHD boundary layer system \eqref{3dimmhd}   is  
\begin{eqnarray*}
	\left\{
	\begin{aligned}
		&\big(\partial_t  +u\partial_x +w\partial_z -\nu\partial_z^2\big) u-(f\partial_x+h\partial_z )f=0,\\
		&\partial_tf+\partial_z(wf-uh)-\mu\partial_z^2f=0,\\
		&\partial_th-\partial_x(wf-uh)-\mu\partial_z^2h=0,
	\end{aligned}
	\right.
\end{eqnarray*}
with the divergence free 
 and  initial-boundary conditions
\begin{equation}\label{inbc}
	\left\{
	\begin{aligned}
		&\partial_xu+\partial_zw=\partial_xf+\partial_zh=0,\\
		& (u,w)|_{z=0}=(\partial_zf,h)|_{z=0}=(0,0), \quad (u, f)|_{z\rightarrow +\infty}=(0,0),\\
		& (u,f)|_{t=0}=(u_0, f_0).
	\end{aligned}
	\right.
\end{equation}
By \eqref{inbc},  we can rewrite  \eqref{3dimmhd} as
\begin{equation}
	\label{mhd+}
\left\{
	\begin{aligned}
		&\big(\partial_t  +u\partial_x +w\partial_z -\nu\partial_z^2\big) u=\xi,\\
		&\big(\partial_t  +u\partial_x +w\partial_z   -\mu\partial_z^2\big) f= \eta,\\
		&\big(\partial_t  +u\partial_x +w\partial_z   -\mu\partial_z^2\big) h= f\partial_xw -h\partial_xu,
	\end{aligned}
	\right.
\end{equation}
where 
\begin{equation}\label{varpsi}
	  \xi=(f\partial_x+h\partial_z)f,\quad \eta=(f\partial_x+h\partial_z)u.
\end{equation}  
Note 
\begin{eqnarray*}
	w(t,x,z)=-\int_0^z \partial_x u(t,x,\tilde z) d\tilde z,\quad h(t,x,z)=-\int_0^z \partial_x f(t,x,\tilde z) d\tilde z .
\end{eqnarray*}
 We remark that the  equation for $h$   in \eqref{mhd+}  can be derived from the one for $f$
 and    the main difficulty in analysis is the loss of $x$-derivatives in the two terms $w$ and $h.$    

The existence and uniqueness theory for \eqref{mhd+} can be stated as follows. 

\begin{theorem}\label{th2d}
	Suppose $(u_{0}, f_{0}) \in X_{2\rho_0, \sigma,8}$  for some  $1<\sigma\leq 3/2$ and $0<\rho_0 \leq 1$, compatible with  the boundary condition in \eqref{inbc}.	 Then the system \eqref{mhd+} with the condition \eqref{inbc}, admits a unique solution $(u,f) \in L^\infty\big([0,T];~X_{\rho,\sigma,4}\big)$ for some $T>0$ and some $0<\rho<2\rho_0.$
	 \end{theorem}

The main part of the proof of Theorem \ref{th2d}  will be given in  Subsections \ref{subu}-\ref{secno} for
proving  the 
a priori estimate stated in Subsection \ref{subap}. 

\medskip

\noindent {\bf Notations}.  Throughout this section  we will  use  $\norm{\cdot}_{L^2}$ and $\inner{\cdot,\ \cdot}_{L^2}$ to denote the norm and inner product of  $L^2=L^2(\mathbb R_+^2),$  and use the notations $\norm{\cdot}_{L^2(\mathbb R_{x})}$ and  $\inner{\cdot,\ \cdot}_{L^2(\mathbb R_{x})}$   when the variable is specified.   Similar notations will be used  for $L^\infty.$  Moreover,   we   use  $L_{x}^p(L_z^q)=L^p\inner{\mathbb R; L^q(\mathbb R_+)}$ for the classical Sobolev space.

   \subsection{A priori estimate}\label{subap}
 
  	 Let   $(u,f)\in L^\infty\inner{[0, T];~X_{\rho_0,\sigma,4}}$  be a solution to  the boundary layer system  \eqref{mhd+} with initial datum $(u_0, f_0)\in X_{2\rho_0,\sigma,8}$ for some  $0<\rho_0\leq 1$ and $1< \sigma\leq 3/2$, recalling    $X_{\rho,\sigma, N}$  is the Gevrey function space  given in  Definition \ref{defgev}.  Moreover, suppose $(\partial_t^iu,\partial_t^if)\in L^\infty\inner{[0, T];~X_{\rho_0,\sigma,4-i}}$ for each $i\leq 4$.  This subsection
  	  together with the following Subsections \ref{subu}-\ref{secno} aim to close the a priori estimate on $u$ and $f$.   
  	  For this,  we first introduce some auxiliary functions defined below.  Some of these functions
  	   were  given in \cite{lmy} for the study on Prandtl equation  inspired  by the work  \cite{MR3925144}.  
  	 
   Let 
 $\mathcal U$ be  a solution to the linear initial-boundary problem
 		\begin{eqnarray}\label{mau}
 		\left\{
 		\begin{aligned}
&   \big(\partial_t+u\partial_x +w\partial_z-\nu\partial_z^2\big)    \int_0^z\mathcal U  d\tilde z  =  -\partial_x w,\\
& \mathcal U|_{t=0}=0, \quad \partial_z\mathcal U|_{z=0}=\mathcal U|_{z\rightarrow+\infty}=0,
   \end{aligned}
   \right.
	\end{eqnarray} 
	where   $\int_0^z\mathcal U d\tilde z=\int_0^z\mathcal U(t,x,\tilde z) d\tilde z$.  
In addition, we define 
 \begin{equation}\label{lamd}
	\lambda=\partial_xu-(\partial_zu)\int_0^z\mathcal U d\tilde z,\quad \delta=\partial_xf-(\partial_zf)\int_0^z\mathcal U d\tilde z.
\end{equation} 
 Note the existence  of solution to  the initial-boundary problem \eqref{mau} follows from the standard parabolic theory.
  With these  functions and $\xi,\eta$ defined in \eqref{varpsi},  set
   \begin{eqnarray*}
 	\vec a=(u,f, \mathcal U,\lambda, \delta, \xi,\eta ).
 \end{eqnarray*}
And then we define the following Gevrey norm on $\vec a$.
 
 \begin{definition}\label{deoa}
Let $\vec a$ be given above,  define  
 \begin{equation}\label{devea}
\begin{aligned}
	  & \abs{\vec a }_{\rho,\sigma}    =    \sup_{\stackrel{i+j\leq 4}{ m+i+j\geq 7}}   \frac{\rho^{m+i+j-7}}{[\inner{m+i+j-7}!]^{\sigma}}\inner{     \norm{ \partial_t^i\partial_x^m \partial_z^ju}_{L^2}+   \norm{ \partial_t^i\partial_x^m\partial_z^j f}_{L^2}}+\sup_{\stackrel{i\leq 4}{ m+i\geq 6}}   \frac{\rho^{m+i-6}}{[\inner{m+i-6}!]^{\sigma}}     \norm{ \partial_t^i\partial_x^m\mathcal U}_{L^2} \\
	   &\quad+\sup_{\stackrel{i\leq 4}{ m+i\geq 6}}    \frac{\rho^{m+i-6}}{[\inner{m+i-6}!]^{\sigma}}   \inner{    m^{1 / 2}  \norm{ \partial_t^i\partial_x^m \lambda}_{L^2}+m^{1 / 2}  \norm{ \partial_t^i\partial_x^m \delta}_{L^2}}\\
	   &\quad +  \sup_{\stackrel{i\leq 4}{ m+i\geq 6}}   \frac{ \rho^{m+i-6}}{ [\inner{m+i-6}!]^{\sigma}}  \inner{m\norm{  \comi z^\ell \partial_t^i\partial_x^m\xi }_{L^2}+m\norm{  \comi z^\ell \partial_t^i\partial_x^m\eta }_{L^2}}  
	 \\
 		&\quad +\sup_{\stackrel{i+j\leq 4}{ m+i+j\leq 6}}    \inner{     \norm{ \partial_t^i \partial_x^m\partial_z^j u}_{L^2}+   \norm{ \partial_t^i\partial_x^m \partial_z^jf}_{L^2}}+	\sup_{\stackrel{i\leq 4}{ m+i\leq 5}}     \norm{\partial_t^i\partial_x^m  \mathcal U}_{L^2}      
	 		   \\
 		& \quad+	\sup_{\stackrel{i\leq 4}{ m+i\leq 5}}     \big(     \norm{  \partial_t^i\partial_x^m\lambda }_{L^2}  +   \norm{  \partial_t^i\partial_x^m\delta }_{L^2}+  \norm{  \comi z^\ell \partial_t^i\partial_x^m\xi }_{L^2}+ \norm{ \comi z^\ell \partial_t^i \partial_x^m\eta }_{L^2}  
	 		   \big),
	 		   	 	\end{aligned}
	 \end{equation}
	  	where  the number $\ell$ is given in  Definition \ref{defgev}. 
   \end{definition}
 
 \begin{remark} Note that we have different powers of  $m$ for the $L^2$ norms  of  the $m^{th}$ order derivatives  $\partial_x^m\mathcal U,\partial_x^m\mathcal \lambda$ and  $\partial_x^m\xi$.  This is motivated by the following relations between these functions:  
  \begin{equation}\label{ulm}
 \left\{
 \begin{aligned}
&\big(\partial_t+u\partial_x +w\partial_z-\nu\partial_z^2\big) \mathcal U  
=   \partial_x\lambda +\textrm{l.o.t.},\\
& \big(\partial_t + u\partial_x   	 +w\partial_z   -\nu\partial _{z}^2\big)  \lambda  
 	 =\partial_x\xi +\textrm{l.o.t.},
\end{aligned}
\right. 	
\end{equation} 
where $\textrm{l.o.t.}  $ refers to lower order terms. Formally,  there is  one order derivative loss
 in both equations  of \eqref{ulm}.    However,  if $\xi$ can be estimated,   then we  lose only $2/3$ order rather than one order derivative  by  treating  $\lambda$ and $\xi$ as  $-1/3$ and $-2/3$  order derivatives of $\mathcal U$. This corresponds to the different powers of $m$ before these auxiliary functions in the definition \eqref{devea}.  
 \end{remark}
 
 \begin{remark}
 	As in \cite{lmy}  the auxiliary function $\mathcal U$ is used to overcome the loss of derivatives in $w$. And to overcome the loss of derivative in $h$, we observe a new cancellation  mechanism for the magnetic convection term $\xi=(f\partial_x +h\partial_z)f$ and this enables us to close the a priori estimate.
 \end{remark}

 Now we state the main result concerning the a priori estimate.  Without loss of generality we only consider the case when the Gevrey index $\sigma=3/2$, and the argument  works with slight modification for  $1<\sigma< 3/2$  (see Subsection \ref{subpr}).

 \begin{assumption}\label{assmain}
	Let  $X_{\rho,\sigma}$  be the Gevrey function space  given in  Definition \ref{defgev}.    Suppose   $(u,f)\in L^\infty\inner{[0, T];~X_{\rho_0,\sigma,4}}$ with some  $0<\rho_0\leq 1$ and $\sigma=3/2$  is a solution to  the boundary layer system  \eqref{mhd+} equipped  with  the condition \eqref{inbc}, where   the initial datum $(u_0, f_0)\in X_{2\rho_0,\sigma,8}$.      Without loss of generality we may assume $T\leq 1$.  Moreover,  we suppose $(\partial_t^iu,\partial_t^if)\in L^\infty\inner{[0, T];~X_{\rho_0,\sigma,4-i}}$ for  $1\le i\leq 4$ and  there exists a constant $C_*$     such that, for any $  t\in[0,T]$,
   \begin{equation}\label{condi1}
\quad  \sup_{\stackrel{  i+ j\leq 4}{k+i+j\leq 10}}  \Big(\big\|\comi z^{\ell+j} \partial_t^i\partial_x^k \partial_z^j   u(t)\big\|_{L^2}+\big\|\comi z^{\ell+j} \partial_t^i\partial_x^k \partial_z^j  f(t)\big\|_{L^2}\Big) \leq  C_*,
\end{equation}
where the constant $C_*\geq 1$  depends only on $\norm{(u_0,f_0)}_{2\rho_0, \sigma,8}$,  the Sobolev embedding  constants and the numbers $\rho_0, \sigma, \ell$ that are given in Definition \ref{defgev}. 
\end{assumption}

\begin{theorem}\label{apriori}
Let $\abs{\vec a}_{\rho,\sigma}$ be given in \eqref{devea}.  Under Assumption \ref{assmain}, there exist
     two constants   $C_1, C_2\geq 1,$     such that  for any pair $(\rho,\tilde\rho)$ with   $0<\rho<\tilde \rho<\rho_0$,   the following estimate
  \begin{equation} \label{aes}
  \begin{aligned}
 	\abs{\vec a (t)}_{\rho,\sigma}^2 \leq & C_1 \inner{\norm{(u_0, f_0)}_{2\rho_0, \sigma,8}^2+\norm{(u_0, f_0)}_{2\rho_0, \sigma, 8}^4}  \\ 
 	& +e^{C_2C_*^2}\bigg( \int_{0}^{t} \inner{\abs{\vec a(s)}_{\rho,\sigma}^2+\abs{\vec a(s)}_{\rho,\sigma}^4} \,ds+  \int_{0}^{t}\frac{ \abs{\vec a(s)}_{\tilde\rho,\sigma}^2}{\tilde \rho-\rho}\,ds\bigg)
 	 \end{aligned}
\end{equation}	 
 holds  for any $t\in[0,T],$   where the constants $C_1$ and $C_2$ depend  only on the Sobolev embedding  constants and the numbers $\rho_0, \sigma, \ell$  given in Definition \ref{defgev}. Both   $C_1$ and $C_2$  are independent of the constant $C_*$ given in \eqref{condi1}.  
 \end{theorem}

The rest  of this section  devotes to the proof of this  a  priori estimate.  We will proceed in the  following Subsections \ref{subu}-\ref{secno}   to derive the estimates on   the terms involved in the definition \eqref{devea} of $\abs{\vec{a}}_{\rho,\sigma}.$

 To simplify the notation,     from now on the   capital letter $C$  denotes some generic constant that may vary from line to line that depends  only on the Sobolev embedding constants and the numbers $ \rho_0, \sigma, \ell$ given in Definition \ref{defgev}  but is  independent of the constant $C_*$   in   \eqref{condi1} and the order of differentiation denoted by $m$.

 \subsection{Tangential derivatives of $\mathcal U$}\label{subu}

For the tangential derivatives of $\mathcal U$  defined in \eqref{mau}, we have the following estimate.

\begin{proposition}\label{prpu}
	Under Assumption \ref{assmain} we have,  for any    $t\in[0,T]$ and 
   for any pair $(\rho,\tilde\rho)$ with   $0<\rho<\tilde \rho<\rho_0\leq 1$,  
     \begin{multline*}
		 	\sup_{m\geq 6}\frac{\rho^{2(m-6)}}{   [\inner{m-6}!]^{2\sigma}}   \norm{\partial_x^{m} \mathcal U(t)}_{L^2}^2 + \sup_{m \leq 5} \norm{\partial_x^{m} \mathcal U(t)}_{L^2}^2 \leq  CC_*\bigg(        \int_0^{t}   \big( \abs{\vec a(s)}_{ \rho,\sigma}^2+\abs{\vec a(s)}_{ \rho,\sigma}^4\big)   \,ds+ \int_0^{t}  \frac{  \abs{\vec a(s)}_{ \tilde\rho,\sigma}^2}{\tilde\rho-\rho}\,ds\bigg),
\end{multline*}
where $C_*\geq 1$ is the constant given in \eqref{condi1}. 
\end{proposition}

\begin{proof}
	  We apply $\partial_z$ to \eqref{mau} and then use  the representation \eqref{lamd} of $\lambda$   to get
 \begin{eqnarray*}
	 \big(\partial_t+u\partial_x +w\partial_z-\nu\partial_z^2\big) \mathcal U =\partial_x\lambda +(\partial_x\partial_zu)\int_0^z\mathcal U d\tilde z  +(\partial_xu)\mathcal U.
\end{eqnarray*}	
	Then applying $\partial_x^{m}$ to the above equation yields
	\begin{eqnarray*}
	 \big(\partial_t+u\partial_x +w\partial_z-\nu\partial_z^2\big) \partial_x^{m} \mathcal U
	 &= &\partial_x^{m+1}\lambda-\sum_{j=1}^{m}{m\choose j}\Big[(\partial_x^j u) \partial_{x}^{m-j+1}\mathcal U  +(\partial_x^j w)\partial_{x}^{m-j}\partial_z\mathcal U\Big] 
	\\   &&+\partial_x^{m}\Big[(\partial_x\partial_zu)\int_0^z\mathcal U d\tilde z  +(\partial_xu)\mathcal U\Big]. 
\end{eqnarray*} 
We take the scalar product with $\partial_x^m\mathcal U$ on the both sides of this equation. Since  $\mathcal U|_{t=0}=\partial_z\mathcal U|_{z=0}=0$, it holds
\begin{eqnarray}\label{uma+}
\begin{aligned}
	&\frac{1}{2}\norm{\partial_x^{m} \mathcal U(t)}_{L^2}^2 +\nu\int_0^{t} \norm{\partial_z\partial_x^{m} \mathcal U(s)}_{L^2}^2ds =\int_0^{t} \Big( \partial_x^{m+1}\lambda,\ \partial_x^{m} \mathcal U\Big)_{L^2}ds\\
	&\quad -\int_0^{t} \Big( \sum_{j=1}^{m}{m\choose j}\Big[(\partial_x^j u) \partial_{x}^{m-j+1}\mathcal U +(\partial_x^j w)\partial_{x}^{m-j}\partial_z\mathcal U\Big],\ \partial_x^{m} \mathcal U \Big)_{L^2}ds\\
	&\qquad +\int_0^{t} \Big(  \partial_x^{m}\Big[(\partial_x\partial_zu )\int_0^z\mathcal U d\tilde z  +(\partial_xu )\mathcal U\Big],\ \partial_x^{m} \mathcal U\Big)_{L^2}ds.
	\end{aligned}
\end{eqnarray}
It remains to derive the upper bound for the terms on the right side in the above equation. 	From definition \eqref{devea} of $\abs{\vec a}_{\rho,\sigma}$,   it  follows
that, for any $ 0< r\leq \rho_0$ and any  $ j\geq 6$, 
 \begin{equation}\label{exi}
  \norm{ \partial_x^j\mathcal U }_{L^2}+  j^{1/2}\inner{\norm{ \partial_x^j\lambda }_{L^2} +\norm{ \partial_x^j\delta}_{L^2} }\leq 
	 \frac{   [\inner{j-6}!]^{ \sigma}}{r^{ (j-6)}}\abs{\vec a}_{r,\sigma}. \end{equation}
	 When $\sigma=3/2,$ we have
	 \begin{eqnarray*}
	 \begin{aligned}
	 	\int_0^{t} \big( \partial_x^{m+1}\lambda,\ \partial_x^{m} \mathcal U \big)_{L^2}ds &\leq  \int_0^{t} m^{-1/2}   \frac{   [\inner{m-5}!]^{ \sigma}}{\tilde\rho^{ (m-5)}} \frac{   [\inner{m-6}!]^{ \sigma}}{\tilde\rho^{ (m-6)}}\abs{\vec a(s)}_{\tilde\rho,\sigma}^2   ds \\
	 	&\leq   C\frac{ [(m-6)!]^{ 2\sigma}  }{  \rho^{ 2(m-6)}}  \int_0^{t} \frac{m^{\sigma-1/2} }{\tilde\rho } \Big(\frac{  \rho }{\tilde\rho}\Big)^{ 2(m-6)}\abs{\vec a(s)}_{\tilde\rho,\sigma}^2   ds\\  
	 	& \leq C\frac{ [(m-6)!]^{ 2\sigma}  }{  \rho^{ 2(m-6)}}    \int_0^{t}  \frac{\abs{\vec a(s)}_{\tilde\rho,\sigma} ^2 }{ \tilde\rho-\rho}       ds,
	 	\end{aligned}
	 \end{eqnarray*}
where in 	the last inequality we have used the fact  that    
 for any integer $k\geq 1$ and for any  pair $(\rho,\tilde \rho)$ with $0<\rho<\tilde\rho\leq 1,$  
\begin{equation}
\label{factor}
   k\inner{\frac{\rho}{\tilde\rho}}^k\leq 	\frac{k}{\tilde\rho} \inner{\frac{\rho}{\tilde\rho}}^k\leq\frac{1}{\tilde\rho-\rho}.
\end{equation}
    On the other hand, the following two estimates are proved respectively in Lemma 3.3 and Lemma 3.4 in \cite{lmy}:
   \begin{eqnarray*}
   &&	-\int_0^{t} \Big( \sum_{j=1}^{m}{m\choose j}\Big[(\partial_x^j u) \partial_{x}^{m-j+1}\mathcal U+ (\partial_x^j w)\partial_{x}^{m-j}\partial_z\mathcal U\Big],\ \partial_x^{m} \mathcal U \Big)_{L^2}ds\\
   	&&\leq {\nu\over 2} \int_0^{t} \norm{\partial_z\partial_x^{m} \mathcal U }_{L^2}^2 ds+ C \frac{[\inner{m-6}!]^{2\sigma}}{\rho^{2(m-6)}} \bigg( \int_0^{t}  \big(\abs{\vec a(s)}_{ \rho,\sigma}^3+   \abs{\vec a(s)}_{ \rho,\sigma}^4 \big)ds\bigg)\\
   	&&\quad+ CC_* \frac{[\inner{m-6}!]^{2\sigma}}{\rho^{2(m-6)}}    \int_0^{t}   \frac{\abs{\vec a(s)}_{\tilde \rho,\sigma}^2}{\tilde\rho-\rho} ds,
   \end{eqnarray*}
   and
   \begin{eqnarray*} 
	&&	\int_0^{t} \Big(\partial_x^{m}\Big[(\partial_x\partial_zu)\int_0^z\mathcal U d\tilde z  +(\partial_xu)\mathcal U\Big], \partial_x^{m} \mathcal U\Big)_{L^2} ds\\
	&&	\leq      \frac{\nu}{2}\int_0^{t}\norm{ \partial_z \partial_x^m\mathcal U}_{L^2}^2ds+  C\frac{  [\inner{m-6}!]^{2\sigma}}{\rho^{2(m-6)}}\int_0^{t} \big(  \abs{\vec a(s)}_{ \rho,\sigma}^3+   \abs{\vec a(s)}_{ \rho,\sigma}^4\big)  ds,
	\end{eqnarray*}
with $C_*\geq 1$  the constant in \eqref{condi1}.   Then we combine the above inequalities with \eqref{uma+} to obtain, for any $m\geq 6$ 
 	 \begin{eqnarray*}
 	 		 \frac{\rho^{2(m-6)}}{   [\inner{m-6}!]^{2\sigma}}   \norm{\partial_x^{m} \mathcal U(t)}_{L^2}^2 
 	 		\leq   C  \int_0^{t}  ( \abs{\vec a(s)}_{ \rho,\sigma}^2+\abs{\vec a(s)}_{ \rho,\sigma}^4 )   \,ds+CC_* \int_0^{t}  \frac{  \abs{\vec a(s)}_{ \tilde\rho,\sigma}^2}{\tilde\rho-\rho}\,ds.
 	 \end{eqnarray*}
 	 The estimate for $m\leq 5$ is straightforward. 
 	Thus the proof of Proposition \ref{prpu} is completed.
 	  \end{proof}

 \subsection{Tangential derivatives of $u$ and $f$}
 
For the tangential derivatives  of  $u,f$, we have the following estimate.

 \begin{proposition}\label{propuv++++}
 Under Assumption \ref{assmain},  
 	  for any    $t\in[0,T]$ and  
  any pair $\inner{\rho,\tilde\rho}$ with $0<\rho<\tilde\rho<\rho_0\leq 1$,    we have
    	 \begin{eqnarray*}	
  \begin{aligned}
&\sup_{m \geq 7}  \frac{\rho^{2(m-7)}} {  [(m-7)!]^{2\sigma}  } \norm{\comi z^{\ell}\partial_x^m u (t)}_{L^2}^2 
 + \sup_{m \leq 6}    \norm{\comi z^{\ell}\partial_x^m u (t)}_{L^2}^2 \\
& \quad+\sup_{m \geq 7}  \frac{\rho^{2(m-7)}} {  [(m-7)!]^{2\sigma}  }\int_0^t\norm{\comi z^{\ell}\partial_z \partial_x^m u (s)}_{L^2}^2ds+\sup_{m \leq 6}   \int_0^t\norm{\comi z^{\ell}\partial_z \partial_x^m u (s)}_{L^2}^2ds\\	
 	&	 \leq  C \norm{(u_0,f_0)}_{2\rho_0,\sigma, 8}^2 + C  C_*^3\bigg(   \int_0^{t}  \big(\abs{\vec a(s)}_{\rho,\sigma}^2+\abs{\vec a(s)}_{\rho,\sigma}^4 \big)ds
 		+   	\int_0^{t}   \frac{ \abs{\vec a(s)}_{\tilde\rho,\sigma}^2}{\tilde\rho-\rho} ds\bigg),
 		\end{aligned}
 	\end{eqnarray*}
 	where $C_*\geq 1 $ is the constant given in \eqref{condi1}.
 	Similarly,  the same upper bound  holds when $\partial_x^m u$ is replaced by $\partial_x^m f$. 
 \end{proposition}

\begin{proof}
   Applying $\partial_x^m$ to the first equation in \eqref{mhd+}   gives
  \begin{eqnarray}\label{zu}
\big(\partial_t+u\partial_x +w\partial_z-\nu\partial_z^2\big)  \partial_x^m  u = -(\partial_x^mw)\partial_z u+\partial_x^m\xi +F_{m}
    	\end{eqnarray}
    	with
    	\begin{equation*} 
		F_{m}=-\sum_{j=1}^{m}{{m}\choose j}  (\partial_x^j u) \partial_x^{m-j+1} u   -\sum_{j=1}^{m-1}{{m}\choose j}  (\partial_x^j w) \partial_x^{m-j} \partial_z u.
	\end{equation*}
	On the other hand,  applying $(\partial_z u)\partial_x^{m-1}$ to \eqref{mau}   yields
	\begin{equation}\label{eqint}
		\big(\partial_t+u\partial_x +w\partial_z-\nu\partial_z^2\big) (\partial_z u) \int_0^z \partial_x^{m-1} \mathcal U  d\tilde z
		=- (\partial_x^mw) \partial_z u+L_m+(\partial_z \xi )\int_0^z \partial_x^{m-1} \mathcal U  d\tilde z
	\end{equation}
	with 
	\begin{equation*}
	\begin{aligned}
		L_m=-(\partial_z u) \sum_{j=1}^{m-1}{{m-1}\choose j} \Big[(\partial_x^j u) \int_0^z \partial_x^{m-j} \mathcal U d\tilde z + (\partial_x^j w) \partial_{x}^{m-1-j} \mathcal U \Big] -2\nu(\partial_z^2 u)\partial_x^{m-1} \mathcal U.
		\end{aligned}
		\end{equation*}
Subtract the  equation \eqref{eqint} by \eqref{zu} to eliminate the highest order term $(\partial_x^mw)\partial_zu$ and  this   gives the equation for
 \begin{eqnarray}\label{plam}
 	\psi_m \stackrel{\rm def}{ =}  \partial_x^m u-(\partial_zu)\int_0^z\partial_x^{m-1} \mathcal U d\tilde z.
 \end{eqnarray}
 That is, 
 \begin{equation*}  
 		\big(\partial_t+u\partial_x +w\partial_z-\nu\partial_z^2\big) \psi_m 
 		=  \partial_x^m\xi+F_{m}-L_m-(\partial_z \xi )\int_0^z \partial_x^{m-1} \mathcal U  d\tilde z,
 \end{equation*}
 and thus
 \begin{multline*}
 	\big(\partial_t+u\partial_x +w\partial_z-\nu\partial_z^2\big) \comi z^\ell\psi_m  
 		= \comi z^\ell \partial_x^m\xi-\comi z^{\ell}(\partial_z \xi )\int_0^z \partial_x^{m-1} \mathcal U  d\tilde z\\
 		+ \comi z^\ell F_{m}-\comi z^\ell L_m+w(\partial_z \comi z^\ell) \psi_m-\nu(\partial_z^2 \comi z^\ell) \psi_m-2\nu(\partial_z\comi z^\ell)\partial_z\psi_m.
 \end{multline*}
Then we take the scalar product  with $ \comi z^\ell \psi_m$ on both sides of the above equation and 
observe $ \comi z^\ell\psi_m|_{z=0}=0,$ to obtain  
 \begin{equation}\label{dps}
 \begin{aligned}
 &\frac{1}{2} \norm{\comi z^{\ell}\psi_m(t)}_{L^2}^2+\nu\int_0^{t} \big\|\partial_z\big(\comi z^{\ell}\psi_m\big)\big\|_{L^2}^2 ds =\frac{1}{2} \norm{\comi z^{\ell}\psi_m(0)}_{L^2}^2+\int_0^{t} \big(\comi z^\ell \partial_x^{m}\xi,  \  \comi z^{\ell}\psi_m\big)_{L^2}ds\\
&\quad-\int_0^{t} \Big(\comi z^{\ell}(\partial_z \xi )\int_0^z \partial_x^{m-1} \mathcal U  d\tilde z,  \  \comi z^{\ell}\psi_m\Big)_{L^2}ds+\int_0^{t} \Big(\comi z^\ell F_{m}-\comi z^\ell L_m,\  \comi z^{\ell}\psi_m\Big)_{L^2}ds\\
& \quad+\int_0^{t} \Big( w(\partial_z \comi z^\ell) \psi_m-\nu(\partial_z^2 \comi z^\ell) \psi_m-2\nu(\partial_z\comi z^\ell)\partial_z\psi_m,    \comi z^{\ell}\psi_m\Big)_{L^2}ds.
 \end{aligned}
 \end{equation}
 As for the first term on the right side,  since $
 	\comi z^{\ell} \psi_m|_{t=0}=\comi z^{\ell} \partial_x^{m}u_0 
$, we have
\begin{eqnarray*}
 \norm{\comi z^\ell  \psi_m (0)}_{L^2}^2  \leq   \frac{[(m-7)!]^{2\sigma}  }{  (2\rho_0)^{2(m-7)}} \norm{(u_0,f_0)}_{2\rho_0,\sigma}^2\leq \frac{ [(m-7)!]^{2\sigma}  }{  \rho^{2(m-7)}} \norm{(u_0,f_0)}_{2\rho_0,\sigma}^2.
\end{eqnarray*} 
The upper bound for the last three terms on the right side of \eqref{dps} was obtained in \cite{lmy} (see  the proof of \cite[Lemma 4.2]{lmy}); that is,
  \begin{eqnarray*}
 	&&\int_0^{t} \Big(\comi z^\ell F_{m}-\comi z^\ell L_{m}, \  \comi z^{\ell}\psi_m\Big)_{L^2}ds\\ 
 	&&\quad+ \int_0^{t} \Big( w(\partial_z \comi z^\ell) \psi_m-\nu(\partial_z^2 \comi z^\ell) \psi_m-2\nu(\partial_z\comi z^\ell)\partial_z\psi_m, \  \comi z^{\ell}\psi_m\Big)_{L^2}ds\\
   && \leq  CC_*^3 \frac{    [(m-7)!]^{2\sigma} } {\rho^{2(m-7)}}\bigg(\int_0^{t} \inner{\abs{\vec a(s)}_{\rho,\sigma}^2+\abs{\vec a(s)}_{\rho,\sigma}^4}ds  + \int_0^{t} \frac{\abs{\vec a(s)}_{\tilde\rho,\sigma}^2}{\tilde\rho-\rho}ds\bigg).
 \end{eqnarray*}
We omit the detail  and refer to the argument in \cite[Lemma 4.2]{lmy}.   It remains to estimate the second and third  terms on the right of \eqref{dps}.    
From  the definition of $\abs{\vec a}_{r,\sigma}$   it follows that,  for any  $ 0<r\leq \rho_0$ and  any $j\geq 7$,
\begin{equation}\label{emix}
 \norm{\comi z^{\ell}\partial_x^j u}_{L^2}+\norm{\comi z^{\ell}\partial_x^j f}_{L^2}\leq 
	 \frac{   [\inner{j-7}!]^{ \sigma}}{r^{ (j-7)}}\abs{\vec a}_{r,\sigma}, 
	 \end{equation}
	 and
	  \begin{equation}\label{uint+}
 \big\|  \comi z^{-1} \int_0^z   \partial_x^j \mathcal U  d\tilde z\big\|_{L_{x}^2(L_z^\infty)} 
 	\leq C\norm{\partial_x^j\mathcal U}_{L^2}\leq \frac{C[(j-6)!]^\sigma}{r^{j-6}}\abs{\vec a}_{r,\sigma}.
 \end{equation}
 Then  we use the definition \eqref{plam} of $\psi_m$ and the condition  \eqref{condi1} to obtain, for any  $ 0<r\leq \rho_0$ and  any $m\geq 7$,  
 \begin{eqnarray}\label{pm+}
	\norm{\comi z^{\ell}\psi_m}_{L^2}\leq  \norm{\comi z^{\ell}\partial_x^m u}_{L^2}+CC_*\norm{\partial_x^{m-1}\mathcal U}_{L^2} \leq CC_* \frac{   [(m-7)!]^{\sigma}  } {r^{m-7}} \abs{\vec a}_{r,\sigma}.
\end{eqnarray} 
 Moreover, note that
 	\begin{equation}\label{xi}
\forall\ 0<r\leq \rho_0,\ \forall\ j\geq 6,\  	 j \norm{\comi z^{\ell}\partial_x^j \xi}_{L^2} + j \norm{\comi z^{\ell}\partial_x^j \eta}_{L^2}\leq 
	 \frac{   [\inner{j-6}!]^{ \sigma}}{r^{ (j-6)}}\abs{\vec a}_{r,\sigma}, \end{equation}
 from the definition of $|\vec a|_{r,\sigma}$.  The above two inequalities  give
 \begin{eqnarray*}
\begin{aligned}
	\int_0^{t} \Big(\comi z^\ell \partial_x^{m}\xi,   \comi z^{\ell}\psi_m\Big)_{L^2}ds&\leq  CC_*\int_0^{t} \frac{1}{m}\frac{   [(m-6)!]^{\sigma}  } {\tilde\rho^{m-6}}  \frac{   [(m-7)!]^{\sigma}  } {\tilde\rho^{m-7}} \abs{\vec a(s)}_{\tilde\rho,\sigma}^2ds\\
	&\leq C C_*\frac{    [(m-7)!]^{2\sigma} } {\rho^{2(m-7)}}\int_0^{t}  \frac{\abs{\vec a(s)}_{\rho,\sigma}^2}{\tilde\rho-\rho}ds,
	\end{aligned}
\end{eqnarray*}
where in the last inequality we have used  \eqref{factor} and  $\sigma=3/2$.     Finally, using \eqref{uint+} and the condition \eqref{condi1} we have by  recalling $\xi=(f\partial_x+h\partial_z ) f$ that
\begin{eqnarray*}
&&	\big\|\comi z^{\ell}(\partial_z \xi )\int_0^z \partial_x^{m-1} \mathcal U  d\tilde z\big\|_{L^2}\\
&& \leq \norm{\comi z^{\ell+1} \partial_z (f\partial_x+h\partial_z ) f}_{L_x^\infty(L_z ^2)} \big\|\comi z^{-1}\int_0^z \partial_x^{m-1} \mathcal U  d\tilde z\big\|_{L_x^2(L_z^\infty)}	\leq    CC_* ^2\frac{   [(m-7)!]^{\sigma}  } {\rho^{m-7}} \abs{\vec a}_{\rho,\sigma}.
\end{eqnarray*}
This with \eqref{pm+} yields
\begin{eqnarray*}
	-\int_0^{t} \Big(\comi z^{\ell}(\partial_z \xi )\int_0^z \partial_x^{m-1} \mathcal U  d\tilde z,  \  \comi z^{\ell}\psi_m\Big)_{L^2}ds\leq  CC_*^3   \frac{   [(m-7)!]^{2\sigma}  }{\rho^{2(m-7)}}	 \int_0^{t}   \abs{\vec a}_{\rho,\sigma}^2 ds.  
\end{eqnarray*}
Putting the above inequalities  into \eqref{dps}  gives
 \begin{multline} \label{epsm}
	\norm{\comi z^{\ell} \psi_m(t)}_{L^2}^2 +\nu\int_0^{t} \norm{\partial_z\big(\comi z^{\ell} \psi_m\big)}_{L^2}^2 dt
 		\leq   		 \frac{  [(m-7)!]^{2\sigma}  } {\rho^{2(m-7)}}\norm{(u_0,f_0)}_{2\rho_0,\sigma}^2 \\+ C  C_*^3 		 \frac{  [(m-7)!]^{2\sigma}  } {\rho^{2(m-7)}} \bigg(    \int_0^{t}  \big(\abs{\vec a(s)}_{\rho,\sigma}^2+\abs{\vec a(s)}_{\rho,\sigma}^4 \big)ds
 		+  	\int_0^{t}   \frac{ \abs{\vec a(s)}_{\tilde\rho,\sigma}^2}{\tilde\rho-\rho} ds\bigg).
 	\end{multline}
 	Note that
 \begin{multline*}
 		\norm{\comi z^{\ell}\partial_x^m u}_{L^2}^2 \leq 2 \big\|\comi z^{\ell}\psi_m\big\|_{L^2}^2+2\big\|\comi z^{\ell}(\partial_zu)\int_0^z\partial_x^{m-1} \mathcal U d\tilde z\big\|_{L^2}^2
 		\leq 2 \norm{\comi z^{\ell}\psi_m}_{L^2}^2+CC_*^2\norm{\partial_x^{m-1}\mathcal U}_{L^2}^2
 	\end{multline*}
due to the definition \eqref{plam} of $\psi_m$. Hence, the two above estimates together with
 Proposition \ref{prpu} 
give
 	 \begin{eqnarray*}	
   	\norm{\comi z^{\ell}\partial_x^m u (t)}_{L^2}^2  		& \leq &2\frac{  [(m-7)!]^{2\sigma}  } {\rho^{2(m-7)}} \norm{(u_0, f_0)}_{2\rho_0,\sigma}^2\\
   	&&+C  C_*^3  \frac{  [(m-7)!]^{2\sigma}  } {\rho^{2(m-7)}} \bigg(  \int_0^{t}  \big(\abs{\vec a(s)}_{\rho,\sigma}^2+\abs{\vec a(s)}_{\rho,\sigma}^4 \big)ds
 		+    	\int_0^{t}   \frac{ \abs{\vec a(s)}_{\tilde\rho,\sigma}^2}{\tilde\rho-\rho} ds\bigg).
 	\end{eqnarray*}
 	Similarly,
 	 \begin{eqnarray*}
 	 \begin{aligned}	
  & 	\int_0^t\norm{\comi z^{\ell}\partial_z \partial_x^m u  }_{L^2}^2ds
   \leq \int_0^t\norm{\partial_z \big(\comi z^{\ell}\partial_x^m u \big)}_{L^2}^2ds+C	\int_0^t \norm{\comi z^{\ell} \partial_x^m u  }_{L^2}^2ds\\ 
   & \leq2 \int_0^t\norm{\partial_z \big(\comi z^{\ell}\psi_m  \big)}_{L^2}^2ds+CC_*	^2\int_0^t\inner{ \norm{  \partial_x^{m-1} \mathcal U }_{L^2}^2+ \norm{\comi z^{\ell} \partial_x^m u }_{L^2}^2}ds  \\
   	&\leq 2\frac{  [(m-7)!]^{2\sigma}  } {\rho^{2(m-7)}} \norm{(u_0, f_0)}_{2\rho_0,\sigma}^2
   	+C  C_*^3  \frac{  [(m-7)!]^{2\sigma}  } {\rho^{2(m-7)}} \bigg(  \int_0^{t}  \big(\abs{\vec a(s)}_{\rho,\sigma}^2+\abs{\vec a(s)}_{\rho,\sigma}^4 \big)ds
 		+    	\int_0^{t}   \frac{ \abs{\vec a(s)}_{\tilde\rho,\sigma}^2}{\tilde\rho-\rho} ds\bigg),
 		\end{aligned}
 	\end{eqnarray*}
 	where in the last inequality we have used  \eqref{epsm} and the estimates \eqref{emix}-\eqref{uint+}.
  Then we obtain the estimate on  $\partial_x^mu$ when $m\geq 7$, and  the estimate  for  $m\leq 6$ is straightforward. It remains to estimate $\partial_x^mf$. For this,    consider 
  \begin{eqnarray*}
  	\varphi_m= \partial_x^m f-(\partial_zf)\int_0^z\partial_x^{m-1} \mathcal U d\tilde z, 
  \end{eqnarray*}
  which satisfies $\partial_z\varphi_m|_{z=0}=0$ and solves 
    \begin{eqnarray}\label{vm}
  		\big(\partial_t+u\partial_x +w\partial_z-\mu\partial_z^2\big) \varphi_m 
 	 &=& (\mu-\nu) (\partial_zf)\partial_z\partial_x^{m-1}\mathcal U+ \partial_x^m\eta+\tilde F_{m}-\tilde L_m\nonumber\\
 	 &&-\big[\partial_z \eta-(\partial_zu)\partial_xf+(\partial_zf)\partial_xu\big]\int_0^z \partial_x^{m-1} \mathcal U  d\tilde z,
 		\end{eqnarray}
  where
  \begin{equation*} 
		\tilde F_{m}=-\sum_{j=1}^{m}{{m}\choose j}  (\partial_x^j u) \partial_x^{m-j+1} f   -\sum_{j=1}^{m-1}{{m}\choose j}  (\partial_x^j w) \partial_x^{m-j} \partial_z f,
	\end{equation*}
  and
  \begin{equation*}
	\begin{aligned}
		\tilde L_m=-(\partial_z f) \sum_{j=1}^{m-1}{{m-1}\choose j} \Big[(\partial_x^j u) \int_0^z \partial_x^{m-j} \mathcal U d\tilde z + (\partial_x^j w) \partial_{x}^{m-1-j} \mathcal U \Big] 
		-2\mu(\partial_z^2 f)\partial_x^{m-1} \mathcal U.
		\end{aligned}
		\end{equation*}
		Observe 
		\begin{multline*}
			\Big((\mu-\nu) \comi z^\ell (\partial_zf)\partial_z\partial_x^{m-1}\mathcal U, \ \comi z^\ell \varphi_m\Big)_{L^2}
			\leq \frac{1}{2}\big\|\partial_z\big(\comi z^{\ell}\varphi_m\big)\big\|_{L^2}^2+CC_*^2\inner{\norm{\partial_x^{m-1}\mathcal U}_{L^2}^2+\norm{\comi z^\ell \varphi_m}_{L^2}^2},
		\end{multline*}
		and the other terms on the right side of \eqref{vm} can be treated similarly as for  $\partial_x^mu$.  Then the estimate \eqref{epsm} also holds with $\psi_m$  replaced by $\varphi_m$.
     	The proof of the proposition  is completed.  
\end{proof}

   \subsection{Tangential derivatives of $ \xi$ and $\eta$}
We now turn to estimate  the tangential derivatives of $\xi$ and $\eta$ which are defined in
 \eqref{varpsi},  that is, 
  $\xi=f\partial_xf+h\partial_zf$ and $\eta=f\partial_xu+h\partial_zu$.
 
 \begin{proposition}
 \label{prp+3}	  
 Under Assumption \ref{assmain} we have, 
 	  for any  $t\in[0,T]$ and   for  any pair $\inner{\rho,\tilde\rho}$ with  $0<\rho<\tilde\rho< \rho_0\leq 1$,  
 	  \begin{eqnarray*}
 	  \begin{aligned}
		 &	\sup_{m\geq 6}\frac{\rho^{2(m-6)}}{   [\inner{m-6}!]^{2\sigma}}  m^2 \inner{\norm{\partial_x^{m} \xi (t)}_{L^2}^2 +\norm{\partial_x^{m} \eta (t)}_{L^2}^2}	+\sup_{m\leq 5} \inner{\norm{\partial_x^{m} \xi (t)}_{L^2}^2 +\norm{\partial_x^{m} \eta (t)}_{L^2}^2}\\
		 	& \leq C  \inner{\norm{(u_0,f_0)}_{2\rho_0,\sigma,8}^2+\norm{(u_0,f_0)}_{2\rho_0,\sigma,8}^4}+ e^{CC_*^2}\bigg(        \int_0^{t}   \big( \abs{\vec a(s)}_{ \rho,\sigma}^2+\abs{\vec a(s)}_{ \rho,\sigma}^4\big)   \,ds+ \int_0^{t}  \frac{  \abs{\vec a(s)}_{ \tilde\rho,\sigma}^2}{\tilde\rho-\rho}\,ds\bigg),
		 	\end{aligned}
\end{eqnarray*}
where $C_*\geq 1$ is the constant given in \eqref{condi1}.  
    
\end{proposition}

 The proof relies on  a newly  observed cancellation property of $\xi$ and $\eta$.  Precisely, we use  the equations in \eqref{mhd+} for $u, f$ and $h$, to derive the equations  for  $\eta$ and $\xi$:
\begin{eqnarray*}
	\big(\partial_t  +u\partial_x +w\partial_z  -\nu \partial_z^2\big) \eta
	&=&(f\partial_x  +h\partial_z)\xi +2\nu\big[(\partial_xf) \partial_z^2 u-(\partial_zf)\partial_x\partial_z u\big]\\
	&&+(\mu-\nu)\big[(\partial_xu)\partial_z^2 f -(\partial_zu)\partial_x\partial_zf \big],
\end{eqnarray*}
and 
\begin{equation*} 
	\big(\partial_t  +u\partial_x +w\partial_z  -\mu \partial_z^2\big) \xi=(f\partial_x +h\partial_z)\eta+2\mu\big[(\partial_xf) \partial_z^2 f-(\partial_zf)\partial_x\partial_z f\big],
\end{equation*}
where the loss of tangential derivative term $\partial_xw$ is  cancelled.   
Now we apply $\comi z^\ell\partial_x^m$ to the above equations for $\xi$ and $\eta$ to get
\begin{eqnarray}\label{sypv}
\left
\{
\begin{aligned}
&	\big(\partial_t  +u\partial_x +w\partial_z  -\nu \partial_z^2\big)\comi z^\ell \partial_x^m\eta
	=(f\partial_x  +h\partial_z)\comi z^\ell\partial_x^m\xi+P_m,\\
&	\big(\partial_t  +u\partial_x +w\partial_z  -\mu \partial_z^2\big)\comi z^\ell \partial_x^m\xi=(f\partial_x +h\partial_z)\comi z^\ell\partial_x^m\eta+Q_m,
	\end{aligned}
	\right.
\end{eqnarray}
where
\begin{eqnarray*}
\begin{aligned}
P_m=&\comi z^\ell\sum_{j=1}^m{m\choose j}\big[(\partial_x^jf)\partial_x^{m-j+1}\xi+(\partial_x^jh)\partial_x^{m-j}\partial_z\xi\big]-\comi z^\ell\sum_{j=1}^m{m\choose j}\big[(\partial_x^ju)\partial_x^{m-j+1}\eta+(\partial_x^jw)\partial_x^{m-j}\partial_z\eta \big]\\
&+2\nu\comi z^\ell\partial_x^m\big[(\partial_xf) \partial_z^2 u-(\partial_zf)\partial_x\partial_z u\big]+(\mu-\nu)\comi z^\ell\partial_x^m\big[(\partial_xu)\partial_z^2 f   -(\partial_zu)\partial_x\partial_zf\big]\\
&+w(\partial_z \comi z^\ell)\partial_x^m\eta-2\nu (\partial_z \comi z^\ell)\partial_z\partial_x^m\eta-\nu(\partial_z^2\comi z^\ell)\partial_x^m
\eta	-h(\partial_z\comi z^\ell)\partial_x^m\xi,
\end{aligned}
\end{eqnarray*}
and
\begin{eqnarray*}
\begin{aligned}
Q_m=&\comi z^\ell\sum_{j=1}^m{m\choose j}\big[(\partial_x^jf)\partial_x^{m-j+1}\eta+(\partial_x^jh)\partial_x^{m-j}\partial_z\eta \big]-\comi z^\ell\sum_{j=1}^m{m\choose j}\big[(\partial_x^ju)\partial_x^{m-j+1}\xi+(\partial_x^jw)\partial_x^{m-j}\partial_z\xi \big]\\
&	+2\mu\comi z^\ell\partial_x^m\big[(\partial_xf) \partial_z^2 f -(\partial_zf)\partial_x\partial_z f\big]+w(\partial_z \comi z^\ell)\partial_x^m\xi\\
&-2\mu (\partial_z \comi z^\ell)\partial_z\partial_x^m\xi-\mu(\partial_z^2\comi z^\ell)\partial_x^m
\xi	-h(\partial_z\comi z^\ell)\partial_x^m\eta.
\end{aligned}
\end{eqnarray*}
Now we take the inner product with $m^2\comi z^\ell\partial_x^m\eta$ for the first equation in \eqref{sypv} and with $m^2\comi z^\ell\partial_x^m\xi$ for the second one, and then  take summation.   Since $\partial_z\xi|_{z=0}=\eta|_{z=0}=0$ and   the first  terms on the right side of \eqref{sypv} are cancelled by symmetry as well as divergence free condition,   we have
\begin{equation}\label{exif}
\begin{aligned}
&	\frac{m^2}{2} \inner{\norm{\comi z^\ell\partial_x^m\eta(t)}_{L^2}^2+\norm{\comi z^\ell\partial_x^m\xi(t)}_{L^2}^2}+\nu  m^2\int_0^t   \norm{\partial_z\big(\comi z^\ell\partial_x^m\eta\big)}_{L^2}^2ds+ \mu m^2\int_0^t \norm{\partial_z\big(\comi z^\ell\partial_x^m\xi\big)}_{L^2}^2 ds\\
&=\frac{m^2}{2} \inner{\norm{\comi z^\ell\partial_x^m\eta(0)}_{L^2}^2+\norm{\comi z^\ell\partial_x^m\xi(0)}_{L^2}^2
}+m^2\int_0^t   \big(P_m,\  \comi z^\ell\partial_x^m\eta\big)_{L^2} ds +m^2\int_0^t\big(Q_m,\ \comi z^\ell\partial_x^m\xi\big)_{L^2} ds.
\end{aligned}
\end{equation} 
The following lemmas are about the estimation on the  terms  in above equality.  

\begin{lemma}\label{lefx}
 Under Assumption \ref{assmain} we have, 
 	  for any  $t\in[0,T]$ and   for  any pair $\inner{\rho,\tilde\rho}$ with  $0<\rho<\tilde\rho< \rho_0\leq 1$,  
	\begin{eqnarray*}
		&&m^2\int_0^t \Big(\comi z^\ell\sum_{1\leq j\leq m}{m\choose j} (\partial_x^jf)\partial_x^{m-j+1}\xi ,\ \comi z^\ell\partial_x^{m}\eta\Big)_{L^2}ds\\
	&&	\leq  
   	C \frac{    [\inner{m-6}!]^{2\sigma}}{\rho^{2(m-6)}} \bigg( \int_0^{t}   \abs{\vec a(s)}_{ \rho,\sigma}^3  ds+  C_*^2    \int_0^{t}   \frac{\abs{\vec a(s)}_{\tilde \rho,\sigma}^2}{\tilde\rho-\rho} ds\bigg).
	\end{eqnarray*}
\end{lemma}
  \begin{proof}
 	Firstly, note that 
 	 \begin{equation}\label{feu}
 	 \begin{aligned}
	&m\sum_{j=1}^{m}{{m}\choose j}  \norm{ \comi z^\ell (\partial_x^jf)\partial_x^{m-j+1}\xi}_{L^2} \\
	&\leq m	\sum_{j=1}^{[m/2]}{{m}\choose j}  \norm{ \partial_x^j f }_{L^\infty} \norm{\comi z^\ell  \partial_x^{m-j+1}\xi}_{L^2}+m \sum_{j= [m/2]+1}^m {{m}\choose j}  \norm{ \partial_x^j f}_{L_{x}^2(L_z^\infty)}\norm{\comi z^\ell \partial_x^{m-j+1} \xi}_{L_{x}^\infty(L_z^2)},
	\end{aligned}
\end{equation}
where  $[p] $ denotes the largest integer less than or equal to $p.$   By
the following Sobolev embedding  inequalities
\begin{equation*}
	\left\{
	\begin{aligned}
	&\norm{F}_{L^\infty(\mathbb R_{x})}\leq  \sqrt 2\Big(\norm{F}_{L_{x}^2}+\norm{\partial_x F}_{{L_{x}^2}}\Big),\\
	&\norm{F}_{L^\infty} \leq   2\Big(\norm{ F}_{L^2}+\norm{\partial_{x}  F}_{L^2} +\norm{  \partial_{z}F}_{L^2}+\norm{\partial_{x} \partial_{z}F}_{L^2}\Big),
	\end{aligned}
	\right. 
	\end{equation*}
and the  estimates \eqref{emix}-\eqref{xi} as well as  \eqref{condi1}, we  have
\begin{eqnarray}\label{spe}
 &&m	\sum_{j=1}^{[m/2]}{{m}\choose j}  \norm{ \partial_x^j f }_{L^\infty} \norm{  \comi z^\ell\partial_x^{m-j+1}\xi}_{L^2} \nonumber\\
  && \leq   C m \sum_{j=5}^{[m/2]}\frac{m!} {j!(m-j)!} \frac{[(j-5)!]^\sigma}{ \rho^{j-5}} \frac{1}{m-j}\frac{[(m-j-5)!]^\sigma}{\rho^{m-j-5}} \abs{\vec a}_{\rho,\sigma}^2\nonumber \\
  &&+ C C_*m\sum_{1\leq j\leq 4} \frac{m!} {j!(m-j)!} \frac{1}{m-j} \frac{[(m-j-5)!]^\sigma}{ \tilde\rho^{m-j-5}} \abs{\vec a}_{\tilde \rho,\sigma}. 
\end{eqnarray}
Direct calculation shows
 \begin{eqnarray*}
m\sum_{1\leq j\leq 4} \frac{m!} {j!(m-j)!} \frac{1}{m-j} \frac{[(m-j-5)!]^\sigma}{ \tilde\rho^{m-j-5}} \abs{\vec a}_{\tilde \rho,\sigma}  \leq C m  \frac{[(m-6)!]^\sigma}{ \tilde \rho^{m-6}} \abs{\vec a}_{\tilde\rho,\sigma}.
 \end{eqnarray*}
Moreover, by using $m/(m-j)\leq C$ for $j\leq [m/2]$,  we have
\begin{eqnarray*}
\begin{aligned}
& m \sum_{j=5}^{[m/2]}\frac{m!} {j!(m-j)!} \frac{[(j-5)!]^\sigma}{ \rho^{j-5}} \frac{1}{m-j}\frac{[(m-j-5)!]^\sigma}{\rho^{m-j-5}} \abs{\vec a}_{\rho,\sigma}^2  \\
	&\leq  C \frac{   \abs{\vec a}_{\rho,\sigma}^2}{\rho^{m-6}}  \sum_{j=5}^{[m/2]} \frac{m! [(j-5)!]^{\sigma-1} [(m-j-5)!]^{\sigma-1}} {j^5(m-j)^5}    \\
	&
	\leq   C\frac{   \abs{\vec a}_{\rho,\sigma}^2}{\rho^{m-6}} \sum_{j=5}^{[m/2]} \frac{(m-6)! m^6} {j^5m^5 }   \frac{[(m-6)!]^{\sigma-1}}{m^{4(\sigma-1)}} 
	 \leq  C \frac{  [(m-6)!]^{\sigma} } {\rho^{m-6}}\abs{\vec a}_{\rho,\sigma}^2, 
	 \end{aligned}
\end{eqnarray*}
where in the  last inequality we have used
  $\sigma=3/2.$   Combining the above inequalities with \eqref{spe} gives
\begin{equation*}
	 m	\sum_{j=1}^{[m/2]}{{m}\choose j}  \norm{ \partial_x^j f }_{L^\infty} \norm{  \comi z^\ell\partial_x^{m-j+1}\xi}_{L^2}  
	 \leq C \frac{  [(m-6)!]^{\sigma} } {\rho^{m-6}}\abs{\vec a}_{\rho,\sigma}^2+C C_*m  \frac{[(m-6)!]^\sigma}{ \tilde \rho^{m-6}} \abs{\vec a}_{\tilde\rho,\sigma}.
\end{equation*}
  Similarly, recalling $\xi=f\partial_xf+h\partial_zf,$ we have
\begin{eqnarray*}
\begin{aligned}
	&m \sum_{j= [m/2]+1}^m {{m}\choose j}  \norm{ \partial_x^j f}_{L_{x}^2(L_z^\infty)}\norm{\comi z^\ell \partial_x^{m-j+1} \xi}_{L_{x}^\infty(L_z^2)}\\
	& \leq C m \sum_{j= [m/2]+1}^{m-4} \frac{m!} {j!(m-j)!} \frac{[(j-6)!]^\sigma}{ \rho^{j-6}} \frac{1}{m-j}\frac{[(m-j-4)!]^\sigma}{\rho^{m-j-4}} \abs{\vec a}_{\rho,\sigma}^2\\
	&\quad+CC_*^2m\sum_{j= m-3}^m \frac{m!} {j!(m-j)!} \frac{[(j-6)!]^\sigma}{ \tilde\rho^{j-6}}  \abs{\vec a}_{\rho,\sigma}\\
	&\leq C \frac{  [(m-6)!]^{\sigma} } {\rho^{m-6}}\abs{\vec a}_{\rho,\sigma}^2+CC_*^2  m  \frac{[(m-6)!]^\sigma}{ \tilde \rho^{m-6}} \abs{\vec a}_{\tilde\rho,\sigma}.
	\end{aligned}
\end{eqnarray*}
Putting these inequalities into \eqref{feu} gives
\begin{equation}\label{uem}
	m\sum_{j=1}^{m}{{m}\choose j}  \norm{ \comi z^\ell (\partial_x^jf)\partial_x^{m-j+1}\xi}_{L^2}
	\leq C \frac{  [(m-6)!]^{\sigma} } {\rho^{m-6}}\abs{\vec a}_{\rho,\sigma}^2+CC_*^2  m  \frac{[(m-6)!]^\sigma}{ \tilde \rho^{m-6}} \abs{\vec a}_{\tilde\rho,\sigma}.
\end{equation}
 This with \eqref{xi}
gives 
   \begin{eqnarray*}
   \begin{aligned}
  & m^2\int_0^t \Big(\comi z^\ell\sum_{1\leq j\leq m}{m\choose j} (\partial_x^jf)\partial_x^{m-j+1}\xi ,\ \comi z^\ell\partial_x^{m}\eta\Big)_{L^2}ds\\
   &	\leq  \int_0^t m\sum_{j=1}^{m}{{m}\choose j}  \norm{ \comi z^\ell (\partial_x^jf)\partial_x^{m-j+1}\xi}_{L^2} \times \inner{m\norm{\comi z^\ell\partial_x^{m}\eta}_{L^2}}ds\\
   		&\leq 
   	C \frac{    [\inner{m-6}!]^{2\sigma}}{\rho^{2(m-6)}}   \int_0^{t}   \abs{\vec a(s)}_{ \rho,\sigma}^3  ds+  CC_*^2   \frac{    [\inner{m-6}!]^{2\sigma}}{\rho^{2(m-6)}}    \int_0^{t}  m \frac{\rho^{2(m-6)}}{\tilde\rho^{2(m-6)}} \abs{\vec a(s)}_{\tilde \rho,\sigma}^2ds \\
  & 	\leq 
   	C \frac{    [\inner{m-6}!]^{2\sigma}}{\rho^{2(m-6)}} \bigg( \int_0^{t}   \abs{\vec a(s)}_{ \rho,\sigma}^3  ds+  C_*^2    \int_0^{t}   \frac{\abs{\vec a(s)}_{\tilde \rho,\sigma}^2}{\tilde\rho-\rho} ds\bigg),
   	\end{aligned}
   \end{eqnarray*}
 where in   the last inequality we have used \eqref{factor}. 
The proof of the lemma  is completed.  
 \end{proof}
 
 \begin{lemma}\label{lehx}
 Under Assumption \ref{assmain} we have, 
 	  for any  $t\in[0,T]$ and   for  any pair $\inner{\rho,\tilde\rho}$ with  $0<\rho<\tilde\rho< \rho_0\leq 1$,  
 	\begin{eqnarray*}
 	&&	m^2\int_0^t \Big(\comi z^\ell\sum_{1\leq j\leq m}{m\choose j} (\partial_x^jh)\partial_x^{m-j}\partial_z\xi ,\ \comi z^\ell\partial_x^{m}\eta\Big)_{L^2}ds\\
 	&&	\leq   \frac{\nu}{6}m^2\int_0^{t}\big\|\partial_z\big(\comi z^\ell \partial_x^m\eta\big)\big\|_{L^2}^2ds 
	+  C\frac{  [\inner{m-6}!]^{2\sigma}}{\rho^{2(m-6)}}\int_0^{t}  \inner{\abs{\vec a(s)}_{ \rho,\sigma}^3+ \abs{\vec a(s)}_{ \rho,\sigma}^4} ds\\
	&&\qquad +CC_*^2 \frac{  [\inner{m-6}!]^{2\sigma}}{\rho^{2(m-6)}}\int_0^{t}   \frac{  \abs{\vec a(s)}_{\tilde \rho,\sigma}^2}{\tilde\rho-\rho}ds.
 	\end{eqnarray*}
 \end{lemma}
 
 \begin{proof}
 	It follows from integration by parts that
\begin{eqnarray}\label{j1j2}
		m^2\int_0^t \Big(\comi z^\ell\sum_{1\leq j\leq m}{m\choose j} (\partial_x^jh)\partial_x^{m-j}\partial_z\xi ,\ \comi z^\ell\partial_x^{m}\eta\Big)_{L^2}ds\leq \sum_{k=1}^4J_k,
\end{eqnarray}
with
\begin{eqnarray*}
J_1&=&m \int_0^{t} \sum_{j=m-1}^{m}{{m}\choose j}\norm{\comi z^\ell \big(\partial_x^j h\big)\partial_x^{m-j} \partial_z \xi}_{L^2}\times \inner{m \norm{\comi z^\ell  \partial_x^m\eta}_{L^2} }ds,\\
J_2&=&m \int_0^{t} \sum_{j=1}^{m-2}{{m}\choose j}\norm{\comi z^\ell \big(\partial_x^j h\big)\partial_x^{m-j} \xi}_{L^2}\times \inner{m \big\|\partial_z\big(\comi z^\ell \partial_x^m\eta\big)\big\|_{L^2} }ds,\\
J_3&=& m \int_0^{t} \sum_{j=1}^{m-2}{{m}\choose j}\norm{\comi z^\ell  (\partial_x^{j+1} f )\partial_x^{m-j} \xi}_{L^2}\times\inner{m\norm{ \comi z^\ell \partial_x^m\eta}_{L^2}} ds,\\
\\
J_4&=& 2m \int_0^{t} \sum_{j=1}^{m-2}{{m}\choose j}\norm{(\partial_z\comi z^\ell)  (\partial_x^{j} h )\partial_x^{m-j} \xi}_{L^2}\times\inner{m\norm{ \comi z^\ell \partial_x^m\eta}_{L^2}} ds.
\end{eqnarray*}
Note that  $
	\norm{\partial_x^j w}_{L_z^\infty}\leq  C \norm{\comi z^{\ell}\partial_x^{j+1}  u}_{L_z^2} 
$ 
for $\ell>1/2$, and similar estimate holds for   $\partial_x^j h$.  Then  it follows from \eqref{emix}  and \eqref{condi1} that, for any $0<r\leq \rho_0,$  
 \begin{equation*}  
  \norm{ \partial_x^j w}_{L_{x}^2(L_z^\infty)} +\norm{ \partial_x^j h}_{L_{x}^2(L_z^\infty)}  
   \leq
	\left\{
	\begin{aligned}
	&C \frac{   [\inner{j-6}!]^{ \sigma}}{r^{ (j-6)}}\abs{\vec a}_{r,\sigma},\quad {\rm if}~j \geq 6,\\
	&C \abs{\vec a}_{r,\sigma},   \quad {\rm if}~ j \leq 5.
	\end{aligned}
	\right.
\end{equation*} 
Thus as for the proof of  \eqref{uem}, when  $\sigma=3/2,$       we have
\begin{eqnarray*}
	m  \sum_{j=1}^{m-2}{{m}\choose j}\norm{\comi z^\ell \big(\partial_x^j h\big)\partial_x^{m-j} \xi}_{L^2}+m   \sum_{j=1}^{m-2}{{m}\choose j}\norm{\comi z^\ell  (\partial_x^{j+1} f )\partial_x^{m-j} \xi}_{L^2}\leq C\frac{    [\inner{m-6}!]^{\sigma}}{\rho^{m-6}} \abs{\vec a}_{\rho,\sigma}^2.
\end{eqnarray*} 
Thus, by  \eqref{xi}, we have 
\begin{equation*} 
	J_2 +J_3+J_4 \leq  \frac{\nu}{6}m^2\int_0^{t}\big\|\partial_z\big(\comi z^\ell \partial_x^m\eta\big)\big\|_{L^2}^2ds
	+  C\frac{  [\inner{m-6}!]^{2\sigma}}{\rho^{2(m-6)}}\int_0^{t}  \inner{\abs{\vec a(s)}_{ \rho,\sigma}^3+ \abs{\vec a(s)}_{ \rho,\sigma}^4} ds.
\end{equation*}
Finally,  by \eqref{condi1} and \eqref{xi}, direct calculation gives 
\begin{eqnarray*}
	J_1\leq CC_*^2 \frac{  [\inner{m-6}!]^{2\sigma}}{\rho^{2(m-6)}}\int_0^{t}  \frac{  \abs{\vec a(s)}_{\tilde \rho,\sigma}^2}{\tilde\rho-\rho} ds.
\end{eqnarray*}
Combining the above estimates with  \eqref{j1j2} completes the proof  of the lemma.  
 \end{proof}

  \begin{lemma}\label{lex+++}
 Under Assumption \ref{assmain} we have, 
 	  for any  $t\in[0,T]$ and   for  any pair $\inner{\rho,\tilde\rho}$ with  $0<\rho<\tilde\rho< \rho_0\leq 1$,  
 	\begin{eqnarray*}
 	\begin{aligned}
 	&	-m^2\int_0^t \Big(\comi z^\ell\sum_{1\leq j\leq m}{m\choose j}\big[(\partial_x^ju)\partial_x^{m-j+1}\eta+(\partial_x^jw)\partial_x^{m-j}\partial_z\eta \big],\ \comi z^\ell\partial_x^{m}\eta\Big)_{L^2}ds\\
 		&\leq   \frac{\nu}{6}m^2\int_0^{t}\big\|\partial_z \big(\comi z^\ell\partial_x^m\eta\big)\big\|_{L^2}^2ds 
	+  C\frac{  [\inner{m-6}!]^{2\sigma}}{\rho^{2(m-6)}}\int_0^{t}  \inner{\abs{\vec a(s)}_{ \rho,\sigma}^3+ \abs{\vec a(s)}_{ \rho,\sigma}^4} ds\\
	&\quad+CC_*^2 \frac{  [\inner{m-6}!]^{2\sigma}}{\rho^{2(m-6)}}\int_0^{t}   \frac{  \abs{\vec a(s)}_{\tilde \rho,\sigma}^2}{\tilde\rho-\rho}ds.
	\end{aligned}
 	\end{eqnarray*}
 	
 \end{lemma}

 We omit the proof of this lemma because it is almost the
 same as those for Lemmas  \ref{lefx} and \ref{lehx}.

  \begin{lemma}\label{lex}
 Under Assumption \ref{assmain} we have, 
 	  for any  $t\in[0,T]$ and   for  any pair $\inner{\rho,\tilde\rho}$ with  $0<\rho<\tilde\rho< \rho_0\leq 1$,  
 	\begin{eqnarray*}
 	\begin{aligned}
 	&	2\nu m^2\int_0^t \Big(  \comi z^\ell\partial_x^m\big[(\partial_xf) \partial_z^2 u-(\partial_zf)\partial_x\partial_z u\big],\ \comi z^\ell\partial_x^{m}\eta\Big)_{L^2}ds\\
 	&\qquad+ (\mu-\nu) m^2\int_0^t \Big(  \comi z^\ell\partial_x^m\big[(\partial_xu)\partial_z^2 f-(\partial_zu)\partial_x\partial_zf  \big],\ \comi z^\ell\partial_x^{m}\eta\Big)_{L^2}ds\\
 		&\leq   \frac{\nu}{6}m^2\int_0^{t}\big\|\partial_z \big(\comi z^\ell\partial_x^m\eta\big)\big\|_{L^2}^2ds+	C\frac{    [\inner{m-6}!]^{2\sigma}}{\rho^{2(m-6)}}    \norm{(u_0,f_0)}_{2\rho_0,\sigma,8}^2\\ 
	&\quad+  e^{CC_*^2}\frac{  [\inner{m-6}!]^{2\sigma}}{\rho^{2(m-6)}}\bigg(\int_0^{t}  \inner{\abs{\vec a(s)}_{ \rho,\sigma}^2+ \abs{\vec a(s)}_{ \rho,\sigma}^4} ds+\int_0^{t}   \frac{  \abs{\vec a(s)}_{\tilde \rho,\sigma}^2}{\tilde\rho-\rho}ds\bigg).
	\end{aligned}
 	\end{eqnarray*}
 	
 \end{lemma}
 
 \begin{proof} By the definition \eqref{lamd}  of $\lambda$ and $\delta$,  we can derive
 \begin{eqnarray*}
 	(\partial_xf) \partial_z^2 u-(\partial_zf)\partial_x\partial_z u&=&\Big(\delta+(\partial_zf)\int_0^z\mathcal U d\tilde z \Big) \partial_z^2 u	 
 	-(\partial_zf)\partial_z\Big(\lambda+(\partial_zu)\int_0^z\mathcal U d\tilde z\Big)  \\ 
 &=&\delta\partial_z^2 u	-(\partial_zf)\partial_z \lambda   -(\partial_zf)(\partial_zu)\mathcal U. 
 \end{eqnarray*}
 Thus
 \begin{equation}\label{koj}
 		2\nu m^2\int_0^t \Big(  \comi z^\ell\partial_x^m\big[(\partial_xf) \partial_z^2 u-(\partial_zf)\partial_x\partial_z u\big],\ \comi z^\ell\partial_x^{m}\eta\Big)_{L^2}ds\leq \sum_{j=1}^4K_j
 \end{equation}
 with
 \begin{equation*}
 \left\{
 \begin{aligned}
 		K_1&=2\nu m^2\int_0^t \sum_{0\leq j\leq 4}{m\choose j}\Big(  \comi z^\ell   (\partial_x^{j}\delta) \partial_x^{m-j}\partial_z^2 u,\ \comi z^\ell\partial_x^{m}\eta\Big)_{L^2}ds,\\
K_2&=2\nu m^2\int_0^t \sum_{j=5}^m{m\choose j}\norm{\comi z^\ell  (\partial_x^{j}\delta) \partial_x^{m-j}\partial_z^2 u }_{L^2} \norm{\comi z^\ell\partial_x^{m}\eta}_{L^2}ds,\\
K_3 & =-2\nu m^2\int_0^t  \Big(\comi z^\ell   \partial_x^{m} \big[ 	(\partial_zf)\partial_z \lambda   \big] , \ \comi z^\ell\partial_x^{m}\eta\Big)_{L^2}ds,\\
K_4 &=  2\nu m^2\int_0^t  \big\|\comi z^\ell   \partial_x^{m} \big[  (\partial_zf)(\partial_zu)\mathcal U\big]\big\|_{L^2} \norm{\comi  z^\ell\partial_x^{m}\eta}_{L^2}ds.
 \end{aligned}
 \right.
 \end{equation*}
 To estimate $K_j, 1\leq j\leq 4,$ we need the following estimates from \cite[Lemma 5.2]{lmy}: 
  \begin{equation}\label{fmu}
	\forall\ t\in[0,T],\quad \sum_{k\leq 9}\big\|	\comi z^{-\ell}  \int_0^z\partial_x^{k}\mathcal U(t) dz \big\|_{L^2} +\sum_{\stackrel{k+j\leq 8}{0\leq j\leq 2}}\big\|	\partial_x^{k}\partial_z^j \mathcal  U(t)   \big\|_{L^2} \leq e^{CC_*^2}, 
\end{equation}
and
	\begin{equation}\label{fde}
	\forall\ t\in[0,T],\quad  \sum_{\stackrel{k+j\leq 8}{0\leq j\leq 2}}\norm{\partial_x^k\partial_z^j\lambda(t)}_{L^2}   \leq e^{CC_*^2}, 
\end{equation}
 where $C_*\geq 1$ is the constant in  \eqref{condi1}.     Then by \eqref{fmu}, \eqref{exi}, \eqref{emix} and \eqref{xi}  as well as \eqref{condi1},   following  the proof for  Lemma \ref{lefx}, we obtain
 \begin{equation}\label{k2k4}
K_2+K_4\leq  e^{CC_*^2} \frac{    [\inner{m-6}!]^{2\sigma}}{\rho^{2(m-6)}} \Big( \int_0^{t}  (\abs{\vec a(s)}_{ \rho,\sigma}^3+\abs{\vec a(s)}_{ \rho,\sigma}^4)  ds+      \int_0^{t}   \frac{\abs{\vec a(s)}_{\tilde \rho,\sigma}^2}{\tilde\rho-\rho} ds\Big).
 \end{equation}
 As for $K_3$, we first write it as 
 \begin{eqnarray*}
 \begin{aligned}
 	K_3& \leq  2\nu m^2\int_0^t \sum_{0\leq j\leq m-5}{m \choose j} \norm{\comi z^\ell      	(\partial_x^{j}\partial_zf)\partial_x^{m-j}\lambda   }_{L^2} \big\|\partial_z\big(\comi z^\ell\partial_x^{m} \eta\big)\big\|_{L^2}ds\\
 &\quad+	2\nu m^2\int_0^t \sum_{0\leq j\leq m-5}{m \choose j} \norm{\comi z^\ell    (\partial_x^{j}\partial_z^2f)\partial_x^{m-j}\lambda    }_{L^2} \norm{\comi z^\ell\partial_x^{m}  \eta}_{L^2}ds\\
 &\quad+	2\nu  m^2\int_0^t \sum_{0\leq j\leq m-5}{m \choose j} \norm{(\partial_z\comi z^\ell)    (\partial_x^{j}\partial_zf)\partial_x^{m-j}\lambda    }_{L^2} \norm{\comi z^\ell\partial_x^{m}  \eta}_{L^2}ds\\
 &\quad + 2\nu m^2\int_0^t \sum_{m-4\leq j\leq m}{m \choose j} \norm{\comi z^\ell      	(\partial_x^{j}\partial_zf)\partial_x^{m-j}\partial_z\lambda   }_{L^2} \norm{\comi z^\ell\partial_x^{m} \eta}_{L^2}ds.
 \end{aligned}
 \end{eqnarray*}
 Then by \eqref{fde}, it holds that 
  \begin{eqnarray}\label{k3}
 K_3 &\leq& \frac{\nu}{24}m^2\int_0^{t}\big\|\partial_z\big(\comi z^\ell  \partial_x^m\eta\big)\big\|_{L^2}^2ds\nonumber\\
 &+&e^{CC_* ^2} \frac{    [\inner{m-6}!]^{2\sigma}}{\rho^{2(m-6)}} \bigg( \int_0^{t}   \big(\abs{\vec a(s)}_{ \rho,\sigma}^3+\abs{\vec a(s)}_{ \rho,\sigma}^4\big)  ds+    \int_0^{t}   \frac{\abs{\vec a(s)}_{\tilde \rho,\sigma}^2}{\tilde\rho-\rho} ds\bigg).
 \end{eqnarray}
 It remains to estimate $K_1$. Again, note that 
 \begin{eqnarray*}
 \begin{aligned}
 	K_1&\leq 2\nu m^2\int_0^t \sum_{0\leq j\leq 4}{m\choose j}\norm{ \comi z^\ell   (\partial_x^{j} \delta) \partial_x^{m-j}\partial_z  u}_{L^2}\big\|\partial_z\big(\comi z^\ell\partial_x^{m}\eta\big)\big\|_{L^2}ds\\
 	&\quad+2\nu m^2\int_0^t \sum_{0\leq j\leq 4}{m\choose j}\norm{ \comi z^\ell   (\partial_x^{j}\partial_z\delta) \partial_x^{m-j}\partial_z  u}_{L^2}\norm{ \comi z^\ell\partial_x^{m}\eta}_{L^2}ds\\
 &\quad	+2\nu m^2\int_0^t \sum_{0\leq j\leq 4}{m\choose j}\norm{ (\partial_z\comi z^\ell)   (\partial_x^{j} \delta) \partial_x^{m-j}\partial_z  u}_{L^2}\norm{ \comi z^\ell\partial_x^{m}\eta}_{L^2}ds.
 \end{aligned}
 \end{eqnarray*}
Observe that the estimate \eqref{fde}  holds when $\lambda$ is replaced by $\delta$. Thus
  direct calculation shows that 
 \begin{eqnarray*}
 &&K_1\leq 	 \frac{\nu}{24}m^2\int_0^{t}\big\|\partial_z\big(\comi z^\ell\partial_x^{m}\eta\big)\big\|_{L^2}^2ds\\
 &&\qquad+  e^{CC_*^2}  m^2 \sum_{0\leq j\leq 4}\Big(\frac{m!}{j!(m-j)!}\Big)^2\int_0^t \norm{ \comi z^\ell    \partial_x^{m-j}\partial_z  u}_{L^2}^2 ds+e^{CC_*^2} \frac{    [\inner{m-6}!]^{2\sigma}}{\rho^{2(m-6)}}         \int_0^{t}   \frac{\abs{\vec a(s)}_{\tilde \rho,\sigma}^2}{\tilde\rho-\rho} ds.
 \end{eqnarray*}
 As for the second term on the right side,   we have
   \begin{eqnarray*}
 \norm{ \comi z^\ell  \partial_x^{m-j}\partial_z  u}_{L^2}^2\leq  \norm{ \comi z^\ell  \partial_x^{m-j-1}\partial_z  u}_{L^2}\norm{ \comi z^\ell  \partial_x^{m-j+1}\partial_z  u}_{L^2}.
   \end{eqnarray*}
Hence
 \begin{eqnarray*}
 \begin{aligned}
 	& e^{CC_*^2}  m^2 \sum_{0\leq j\leq 4}\Big(\frac{m!}{j!(m-j)!}\Big)^2\int_0^t \norm{ \comi z^\ell   (\partial_x^{j} \delta) \partial_x^{m-j}\partial_z  u}_{L^2}^2 ds\\
 	 &\leq   e^{CC_*^2}   m^4 \sum_{0\leq j\leq 4}\Big(\frac{m!}{j!(m-j)!}\Big)^2\int_0^t \norm{ \comi z^\ell     \partial_x^{m-j-1}\partial_z  u}_{L^2}^2  ds\\
 	&\quad  +  \sum_{0\leq j\leq 4}\Big(\frac{m!}{j!(m-j)!}\Big)^2\int_0^t  \norm{ \comi z^\ell     \partial_x^{m-j+1}\partial_z  u}_{L^2} ^2ds\\
 	&\leq    e^{CC_*^2} \frac{    [\inner{m-6}!]^{2\sigma}}{\rho^{2(m-6)}}     \int_0^{t}   \frac{\abs{\vec a(s)}_{\tilde \rho,\sigma}^2}{\tilde\rho-\rho} ds 
 	+C \frac{    [\inner{m-6}!]^{2\sigma}}{\rho^{2(m-6)}}    \norm{(u_0,f_0)}_{2\rho_0,\sigma,8}^2 \\
 	&\quad +  C  C_*^3\frac{    [\inner{m-6}!]^{2\sigma}}{\rho^{2(m-6)}}    \bigg(   \int_0^{t}  \big(\abs{\vec a(s)}_{\rho,\sigma}^2+\abs{\vec a(s)}_{\rho,\sigma}^4 \big)ds
 		+   	\int_0^{t}   \frac{ \abs{\vec a(s)}_{\tilde\rho,\sigma}^2}{\tilde\rho-\rho} ds\bigg),
 	 \end{aligned}
 \end{eqnarray*}
 where we have used  Proposition \ref{propuv++++} to control the second term in the first inequality. Thus,
 we can conclude that 
 \begin{eqnarray*}
 K_1&\leq & \frac{\nu}{24}m^2\int_0^{t}\big\|\partial_z\big(\comi z^\ell\partial_x^{m}\eta\big)\big\|_{L^2}^2ds+	C \frac{    [\inner{m-6}!]^{2\sigma}}{\rho^{2(m-6)}}    \norm{(u_0,f_0)}_{2\rho_0,\sigma}^2 \\
 	  &&+  e^{C  C_*^2}\frac{    [\inner{m-6}!]^{2\sigma}}{\rho^{2(m-6)}}    \bigg(   \int_0^{t}  \big(\abs{\vec a(s)}_{\rho,\sigma}^2+\abs{\vec a(s)}_{\rho,\sigma}^4 \big)ds
 		+   	\int_0^{t}   \frac{ \abs{\vec a(s)}_{\tilde\rho,\sigma}^2}{\tilde\rho-\rho} ds\bigg).
 \end{eqnarray*}
 Putting the above inequality and the estimates \eqref{k2k4}-\eqref{k3} into  \eqref{koj} gives
 the  upper bound for the first term on the left side in Lemma \ref{lex}.
 The second term can be estimated similarly and we omit the detail. Then the  proof of the lemma  is completed. 
  \end{proof}

  \begin{lemma}\label{++lex+++}
 Under Assumption \ref{assmain} we have, 
 	  for any  $t\in[0,T]$ and   for  any pair $\inner{\rho,\tilde\rho}$ with  $0<\rho<\tilde\rho< \rho_0\leq 1$,  
 	\begin{eqnarray*}
 	&& m^2\int_0^t \Big(  w(\partial_z \comi z^\ell)\partial_x^m\eta-2\nu (\partial_z \comi z^\ell)\partial_z\partial_x^m\eta,\ \comi z^\ell\partial_x^{m}\eta\Big)_{L^2}ds\\
 	&&\qquad+m^2\int_0^t \Big(   -\nu(\partial_z^2\comi z^\ell)\partial_x^m
\eta	-h(\partial_z\comi z^\ell)\partial_x^m\xi,\ \comi z^\ell\partial_x^{m}\eta\Big)_{L^2}ds\\
 		&&\leq    C\frac{  [\inner{m-6}!]^{2\sigma}}{\rho^{2(m-6)}}\int_0^{t}  \inner{\abs{\vec a(s)}_{ \rho,\sigma}^2+ \abs{\vec a(s)}_{ \rho,\sigma}^3} ds.
 	\end{eqnarray*}
 	
 \end{lemma}
Since the proof of this lemma follows from  direct calculation and \eqref{xi}, we omit it for brevity. 
And now we are ready to prove Proposition \ref{prp+3}.

 \begin{proof}
 	[Proof of Proposition \ref{prp+3}]
 	In view of the  representation of $P_m$ given in \eqref{sypv},  we combine the estimates in Lemmas \ref{lefx}-\ref{++lex+++} to conclude
 	\begin{multline*}
 		m^2\int_0^t   \inner{P_m,\  \comi z^\ell\partial_x^m\eta}_{L^2} ds
 		\leq \frac{\nu}{2}m^2\int_0^{t}\big\|\partial_z\big(\comi z^\ell\partial_x^{m}\eta\big)\big\|_{L^2}^2ds+	C\frac{    [\inner{m-6}!]^{2\sigma}}{\rho^{2(m-6)}}    \norm{(u_0,f_0)}_{2\rho_0,\sigma,8}^2 \\
 	  +  e^{C  C_*^2}\frac{    [\inner{m-6}!]^{2\sigma}}{\rho^{2(m-6)}}    \bigg(   \int_0^{t}  \big(\abs{\vec a(s)}_{\rho,\sigma}^2+\abs{\vec a(s)}_{\rho,\sigma}^4 \big)ds
 		+   	\int_0^{t}   \frac{ \abs{\vec a(s)}_{\tilde\rho,\sigma}^2}{\tilde\rho-\rho} ds\bigg).
 	\end{multline*}
 	Similar  upper bound holds for 
 	\begin{eqnarray*}
 		m^2\int_0^t\inner{Q_m,\ \comi z^\ell\partial_x^m\xi}_{L^2} ds.
 	\end{eqnarray*} 
 	Then by  \eqref{exif}, we have
 	\begin{multline*}
 		 m^2  \inner{\norm{\comi z^\ell\partial_x^m\eta(t)}_{L^2}^2+\norm{\comi z^\ell\partial_x^m\xi(t)}_{L^2}^2}\\
 		 \leq   m^2  \inner{\norm{\comi z^\ell\partial_x^m\eta(0)}_{L^2}^2+\norm{\comi z^\ell\partial_x^m\xi(0)}_{L^2}^2
}+	C \frac{    [\inner{m-6}!]^{2\sigma}}{\rho^{2(m-6)}}    \norm{(u_0,f_0)}_{2\rho_0,\sigma,8}^2 \\
 	  +  e^{C  C_*^2}\frac{    [\inner{m-6}!]^{2\sigma}}{\rho^{2(m-6)}}    \bigg(   \int_0^{t}  \big(\abs{\vec a(s)}_{\rho,\sigma}^2+\abs{\vec a(s)}_{\rho,\sigma}^4 \big)ds
 		+   	\int_0^{t}   \frac{ \abs{\vec a(s)}_{\tilde\rho,\sigma}^2}{\tilde\rho-\rho} ds\bigg).
 	\end{multline*}
 	 Moreover,  following the argument for proving Lemma \ref{lefx} we can derive 
 	 \begin{multline*}
 	 m \norm{\comi z^\ell\partial_x^m\xi(0)}_{L^2}
 	 	\leq m \sum_{j\leq m}{m\choose j}\inner{\norm{\comi z^\ell(\partial_x^{j}f_0 )\partial_x^{m-j+1}f_0}_{L^2}+\norm{\comi z^\ell(\partial_x^{j}h(0))\partial_x^{m-j}\partial_zf_0}_{L^2}}\\
 	 	\leq C m \frac{    [\inner{m-6}!]^{\sigma}}{(2\rho_0)^{m-6}}  \norm{(u_0,f_0)}_{ 2\rho_0,\sigma,8}^2  \leq C _1\frac{    [\inner{m-6}!]^{\sigma}}{\rho^{m-6}}  \norm{(u_0,f_0)}_{2\rho_0,\sigma,8}^2,
 	 \end{multline*}
 	 where in the last inequality we have used the fact that $\rho\leq\rho_0$ as well as  $h|_{t=0}=-\int_0^z \partial_xf_0d\tilde z$.   Similar upper bound holds for $m \norm{\comi z^\ell\partial_x^m\eta(0)}_{L^2}$.  Thus combining the above inequalities yields the desired estimate in Proposition \ref{prp+3}  for $m\geq 6$.  The estimate for $m\leq 5$ is straightforward so that the proof of the proposition is completed.
 \end{proof}

 \subsection{Tangential derivatives of  $\lambda$ and $\delta$}\label{sublam}
  The estimate on the tangential derivatives of $\lambda$ and $\delta$
  defined in \eqref{lamd} is given in the following proposition.

 \begin{proposition}\label{prla}
 Under Assumption \ref{assmain} we have, 
 	  for any  $t\in[0,T]$ and   for  any pair $\inner{\rho,\tilde\rho}$ with  $0<\rho<\tilde\rho< \rho_0\leq 1$,  
 	  \begin{multline*}
 \sup_{m\geq 6} \frac{\rho^{2(m-6)}}	{  [\inner{m-6}!]^{2\sigma}}m \norm{  \partial_x^m \lambda}_{L^2}^2 + \sup_{m\leq 5}  \norm{ \partial_x^m \lambda}_{L^2}^2 \\
	 \leq   C \norm{(u_0,f_0)}_{2\rho_0,\sigma,8}^2+  C \int_0^{t}  \inner{\abs{\vec a(s)}_{ \rho,\sigma}^3+ \abs{\vec a(s)}_{ \rho,\sigma}^4} ds 
	+e^{CC_*^2}  \int_0^{t}   \frac{  \abs{\vec a(s)}_{\tilde \rho,\sigma}^2}{\tilde\rho-\rho}ds.
  	\end{multline*}
  	The above estimate also holds when $\lambda$ is replaced by $\delta$. 
 \end{proposition}
 
 \begin{proof}
We apply  $\partial_x$ to the equation for $u$ in \eqref{mhd+} and multiply  the equation \eqref{mau}  by $\partial_z u$, and then the subtraction of these two equations gives   the equation for $\lambda$:
 	\begin{eqnarray*} 
	\big(\partial_t  +u\partial_x +w\partial_z  -\nu\partial_z^2\big)\lambda
	= \partial_x\xi  -(\partial_xu)\partial_xu-  (\partial_z\xi )  \int_0^z\mathcal U d\tilde z+2\nu(\partial_z^2 u)\mathcal U.
\end{eqnarray*}
Thus
\begin{eqnarray*}
 	\big(\partial_t+u\partial_x +w\partial_z-\nu\partial_z^2 \big)   \partial_x^m \lambda 
	  &=&  \partial_x^{m+1}\xi- \partial_x^m\Big[(\partial_xu)\partial_xu+(\partial_z\xi )  \int_0^z\mathcal U d\tilde z -2\nu(\partial_z^2 u)\mathcal U\Big]\\
	  &&- \sum_{j=1}^m{m\choose j} \Big[(\partial_x^j u)\partial_x^{m-j+1}\lambda+ (\partial_x^j w)\partial_x^{m-j}\partial_z\lambda\Big].
	\end{eqnarray*}
Taking inner product with $m   \partial_x^m \lambda$ and observing $\lambda|_{z=0}=0,$ we obtain
\begin{eqnarray}\label{dl}
\begin{aligned}
	&\frac{m}{2}\norm{  \partial_x^m \lambda(t)}_{L^2}^2+m\nu\int_0^t\norm{\partial_z  \partial_x^m \lambda(s) }_{L^2}^2ds\\
	&\leq \frac{m}{2}\norm{ \partial_x^m \lambda(0)}_{L^2}^2+  m\int_0^t\big(  \partial_x^{m+1}\xi,\   \partial_x^m \lambda\big)_{L^2}ds\\
	&\quad -  m\int_0^t\Big(  \partial_x^m\Big[(\partial_xu)\partial_xu+(\partial_z\xi )  \int_0^z\mathcal U d\tilde z -2\nu(\partial_z^2 u)\mathcal U\Big],\   \partial_x^m \lambda\Big)_{L^2}ds\\
	  &\quad-m \int_0^t\Big(  
	   \sum_{j=1}^m{m\choose j}  \Big[(\partial_x^j u)\partial_x^{m-j+1}\lambda+ (\partial_x^j w)\partial_x^{m-j}\partial_z\lambda\Big],\  \partial_x^m \lambda\Big)_{L^2}ds.
	  \end{aligned}
\end{eqnarray}
By \eqref{exi} and \eqref{xi}, when $\sigma=3/2$ we have
 \begin{eqnarray*}
 \begin{aligned}
& m\int_0^t\big(  \partial_x^{m+1}\xi,\   \partial_x^m \lambda\big)_{L^2}ds\leq          \int_0^{t}  m^{-1/2} \frac{ [\inner{m-5}!]^{\sigma}}{\tilde\rho^{m-5}}\frac{ [\inner{m-6}!]^{\sigma}}{\tilde\rho^{m-6}} \abs{\vec a(s)}_{\tilde \rho,\sigma}^2ds\\
 &\leq  C     \frac{    [\inner{m-6}!]^{2\sigma}}{\rho^{2(m-6)}}    \int_0^{t} \frac{ m^{\sigma-\frac{1}{2}}}{\tilde\rho} \frac{\rho^{2(m-6)}}{\tilde\rho^{2(m-6)}} \abs{\vec a(s)}_{\tilde \rho,\sigma}^2ds \leq  C     \frac{    [\inner{m-6}!]^{2\sigma}}{\rho^{2(m-6)}}    \int_0^{t}   \frac{ \abs{\vec a(s)}_{\tilde \rho,\sigma}^2}{\tilde\rho-\rho}ds,
 \end{aligned}
 \end{eqnarray*} 
 where in the last inequality we have used  \eqref{factor}.  Moreover, similar to the proofs of  Lemmas \ref{lefx} and \ref{lehx}, we obtain
 \begin{eqnarray*}
 \begin{aligned}
 &-  m\int_0^t\Big(  \partial_x^m\Big[(\partial_xu)\partial_xu+(\partial_z\xi )  \int_0^z\mathcal U d\tilde z -2\nu(\partial_z^2 u)\mathcal U\Big],\   \partial_x^m \lambda\Big)_{L^2}ds\\
	  &\qquad-m \int_0^t\Big(  
	   \sum_{j=1}^m{m\choose j}  \Big[(\partial_x^j u)\partial_x^{m-j+1}\lambda+ (\partial_x^j w)\partial_x^{m-j}\partial_z\lambda\Big],\  \partial_x^m \lambda\Big)_{L^2}ds\\
	   &\leq    \frac{\nu}{2} m \int_0^{t}\norm{ \partial_z  \partial_x^m\lambda}_{L^2}^2ds 
	+  C\frac{  [\inner{m-6}!]^{2\sigma}}{\rho^{2(m-6)}}\int_0^{t}  \inner{\abs{\vec a(s)}_{ \rho,\sigma}^3+ \abs{\vec a(s)}_{ \rho,\sigma}^4} ds+e^{CC_*^2} \frac{  [\inner{m-6}!]^{2\sigma}}{\rho^{2(m-6)}}\int_0^{t}   \frac{  \abs{\vec a(s)}_{\tilde \rho,\sigma}^2}{\tilde\rho-\rho}ds.
	  \end{aligned}
 \end{eqnarray*}
 Combining the above inequalities with \eqref{dl} gives
   \begin{eqnarray*}
  	&&m\norm{  \partial_x^m \lambda}_{L^2}^2 
	 \leq  m\norm{  \partial_x^m \lambda(0)}_{L^2}^2+  C\frac{  [\inner{m-6}!]^{2\sigma}}{\rho^{2(m-6)}}\int_0^{t}  \inner{\abs{\vec a(s)}_{ \rho,\sigma}^3+ \abs{\vec a(s)}_{ \rho,\sigma}^4} ds\\
	&&\qquad +e^{CC_*^2} \frac{  [\inner{m-6}!]^{2\sigma}}{\rho^{2(m-6)}}\int_0^{t}   \frac{  \abs{\vec a(s)}_{\tilde \rho,\sigma}^2}{\tilde\rho-\rho}ds.
  	\end{eqnarray*}
Note that $\lambda|_{t=0}=\partial_xu_0$. Hence 
\begin{eqnarray*}
	m \norm{  \partial_x^m \lambda (0)}_{L^2}^2  \leq   m \frac{[(m-6)!]^{2\sigma}  }{  (2\rho_0)^{2(m-6)}} \norm{(u_0,f_0)}_{2\rho_0,\sigma,8}^2\leq C \frac{ [(m-6)!]^{2\sigma}  }{  \rho^{2(m-6)}} \norm{(u_0,f_0)}_{2\rho_0,\sigma,8}^2
\end{eqnarray*}
because $\rho\leq\rho_0.$
Thus  we obtain the desired estimate on $\lambda$  for $m\geq 6$. Again, the estimate
for $m\leq 5$ is straightforward. Note that the upper bound for $\delta$ can be derived similarly because
there is  no non-zero boundary terms in the integration by parts due to the fact that $\partial_z\delta|_{z=0}=0$.  	The  proof of the proposition is  completed. 
 \end{proof}

 \subsection{Time  derivatives} 
 The estimate  involving  $t$-derivatives is stated as follows.  Note it is only in this estimate
  that we need  the normal derivatives of the initial data $u_0$ and $f_0$ up to the 8th order.   
 
  \begin{proposition}\label{lemtime}
 	 Under Assumption \ref{assmain} we have,  
 	  for any    $t\in[0,T]$ and  
  any pair $\inner{\rho,\tilde\rho}$ with $0<\rho<\tilde\rho<\rho_0\leq 1$,
   	 \begin{eqnarray*}
   	 \begin{aligned}
 & \sup_{\stackrel{1\leq i\leq 4}{ m+i  \geq 7}}  \frac{\rho^{2(m+i-7)}} {  [(m+i-7)!]^{2\sigma}  }  	\inner{\norm{\comi z^{\ell}\partial_t^i\partial_x^m  u (t)}_{L^2}^2+\norm{\comi z^{\ell}\partial_t^i\partial_x^m  f (t)}_{L^2}^2}    \\
  &\qquad+ \sup_{\stackrel{1\leq i\leq 4}{ m+i  \geq 7}}  \frac{\rho^{2(m+i-7)}} {  [(m+i-7)!]^{2\sigma}  } \int_0^t 	\inner{\norm{\comi z^{\ell}\partial_z\partial_t^i\partial_x^m  u }_{L^2}^2+\norm{\comi z^{\ell}\partial_z\partial_t^i\partial_x^m  f  }_{L^2}^2}ds    \\
 &\qquad+\sup_{\stackrel{1\leq i\leq 4}{ m+i\geq 6}}   \frac{\rho^{m+i-6}}{[\inner{m+i-6}!]^{\sigma}}     \norm{ \partial_t^i\partial_x^m\mathcal U(t)}_{L^2} \\
	   &\qquad+\sup_{\stackrel{1\leq i\leq 4}{ m+i\geq 6}}    \frac{\rho^{m+i-6}}{[\inner{m+i-6}!]^{\sigma}}   \inner{    m^{1 / 2}  \norm{ \partial_t^i\partial_x^m \lambda(t)}_{L^2}+m^{1 / 2}  \norm{ \partial_t^i\partial_x^m \delta(t)}_{L^2}}\\
	   &\qquad +  \sup_{\stackrel{1\leq i\leq 4}{ m+i\geq 6}}   \frac{ \rho^{m+i-6}}{ [\inner{m+i-6}!]^{\sigma}}  \inner{m\norm{  \comi z^\ell \partial_t^i\partial_x^m\xi (t)}_{L^2}+m\norm{  \comi z^\ell \partial_t^i\partial_x^m\eta(t) }_{L^2}}  
	 \\
 		&\qquad +\sup_{\stackrel{1\leq i \leq 4}{ m+i \leq 6}}    \inner{     \norm{ \partial_t^i \partial_x^m u(t)}_{L^2}+   \norm{ \partial_t^i \partial_x^m f(t)}_{L^2}}+	\sup_{\stackrel{1\leq i\leq 4}{ m+i\leq 5}}     \norm{\partial_t^i\partial_x^m  \mathcal U(t)}_{L^2}      
	 		   \\
 		& \qquad+	\sup_{\stackrel{1\leq i\leq 4}{ m+i\leq 5}}     \big(     \norm{  \partial_t^i\partial_x^m\lambda }_{L^2}  +   \norm{  \partial_t^i\partial_x^m\delta }_{L^2}+  \norm{  \comi z^\ell \partial_t^i\partial_x^m\xi }_{L^2}+ \norm{ \comi z^\ell \partial_t^i \partial_x^m\eta }_{L^2}  
	 		   \big)\\
 		 &\leq  C\inner{\norm{(u_0,f_0)}_{2\rho_0,\sigma,8}^2+\norm{(u_0,f_0)}_{2\rho_0,\sigma,8}^4}		 +e^{CC_*^2} \bigg(    \int_0^{t}  \big(\abs{\vec a(s)}_{\rho,\sigma}^2+\abs{\vec a(s)}_{\rho,\sigma}^4 \big)ds +
 		  \int_0^{t}  \frac{  \abs{\vec a(s)}_{ \tilde\rho,\sigma}^2}{\tilde\rho-\rho}\,ds\bigg).
 		  \end{aligned}
 	\end{eqnarray*} 
 \end{proposition}
 
 \begin{proof}
 	The proof is similar as those in the previous Subsections \ref{subu}-\ref{sublam}, with the tangential derivatives $\partial_x^m$  replaced by $\partial_t^i\partial_x^m$.  The main difference arises from the    initial data.   Note that
 	\begin{eqnarray*}
 		\partial_t u|_{t=0}=\nu\partial_z^2u_0-u_0\partial_xu_0-w_0\partial_zu_0+f_0\partial_xf_0+h_0\partial_zf_0
 	\end{eqnarray*}
 	with $w_0=-\int_0^z\partial_xu_0d\tilde z$ and $h_0=-\int_0^z\partial_xf_0d\tilde z$.
 	Similar expressions hold for   $\partial_t^i u|_{t=0}$,  $2\leq i\leq 4$ in terms of $u_0$ and $f_0$.   Thus we can control the terms $\partial_t^i u|_{t=0}, 1\leq i\leq 4,$ by the initial data $u_0$ and $f_0$ if
 the  normal derivatives of $u_0$ and  $f_0$ are up to the 8th order  rather than 4-th.    Other than the difference
 in the order of differentiation on the initial data,  there is  no essential difference from the argument in the previous subsections on tangential detivatives. Hence,   we omit the detail of the  proof.
 \end{proof}
 
 \subsection{Normal derivatives of $u$ and $f$} \label{secno}

 It remains to estimate the  normal derivatives of $u$ and $f$ in the a priori estimate
 and it is given in the following proposition.
 
  \begin{proposition}  \label{lemnor}
  Under Assumption \ref{assmain} we have,  
 	  for any    $t\in[0,T]$ and  
  any pair $\inner{\rho,\tilde\rho}$ with $0<\rho<\tilde\rho<\rho_0\leq 1$,
   	 \begin{multline*}	
  \sup_{\stackrel{1\leq i+ j\leq 4}{ m+i+j \geq 7}}  \frac{\rho^{2(m+i+j-7)}} {  [(m+i+j-7)!]^{2\sigma}  }  	\norm{\comi z^{\ell+j}\partial_t^i\partial_x^m \partial_z^ju (t)}_{L^2}^2\\
  +\sup_{\stackrel{1\leq i+ j\leq 4}{ m+i+j \geq 7}}  \frac{\rho^{2(m+i+j-7)}} {  [(m+i+j-7)!]^{2\sigma}  }  \int_0^t\norm{\comi z^{\ell+j}\partial_t^i\partial_x^m \partial_z^{j+1}u (s)}_{L^2}^2ds 
  +\sup_{\stackrel{1\leq i+j\leq 4} {m+i+j \leq 6}}   \norm{\comi z^{\ell+j}\partial_t^i\partial_x^m \partial_z^ju (t)}_{L^2}^2  \\	
 		 \leq  C\inner{\norm{(u_0,f_0)}_{2\rho_0,\sigma,8}^2+\norm{(u_0,f_0)}_{2\rho_0,\sigma,8}^4}
 		 +e^{CC_*^2} \bigg(    \int_0^{t}  \big(\abs{\vec a(s)}_{\rho,\sigma}^2+\abs{\vec a(s)}_{\rho,\sigma}^4 \big)ds +
 		  \int_0^{t}  \frac{  \abs{\vec a(s)}_{ \tilde\rho,\sigma}^2}{\tilde\rho-\rho}\,ds\bigg),
 	\end{multline*} 
 	where $C_*\geq 1$ is the constant given in \eqref{condi1}.
 The above estimate also holds when $\comi z^{\ell+j}\partial_t^i\partial_x^m \partial_z^ju$ is  replaced by $\comi z^{\ell+j}\partial_t^i\partial_x^m \partial_z^jf$.
 \end{proposition}

 \begin{proof}
 We apply induction on $j$, the order of normal derivatives. The validity for $j=0$ follows from Proposition \ref{lemtime}.   Now for given $j\geq 1$ and for any    $i$ and $m$,  suppose the estimate
  	 \begin{equation}	
  	 \label{ina}
  	 \begin{aligned}
  &\norm{\comi z^{\ell+k}\partial_t^i\partial_x^m \partial_z^ku (t)}_{L^2}^2+\norm{\comi z^{\ell+k}\partial_t^i\partial_x^m \partial_z^kf (t)}_{L^2}^2\\
  &\quad+\int_0^t \inner{ \norm{\comi z^{\ell+k}\partial_t^i\partial_x^m \partial_z^{k+1}u (s)}_{L^2}^2+ \norm{\comi z^{\ell+k}\partial_t^i\partial_x^m \partial_z^{k+1}f (s)}_{L^2}^2}ds  \\	
 		& \leq  C\frac{  [(m+i+k-7)!]^{2\sigma}  }{\rho^{2(m+i+k-7)}}  \inner{\norm{(u_0,f_0)}_{2\rho_0,\sigma,8}^2+\norm{(u_0,f_0)}_{2\rho_0,\sigma,8}^4}\\
 		 &\quad+e^{CC_*^2}\frac{  [(m+i+k-7)!]^{2\sigma}  }{\rho^{2(m+i+k-7)}}    \bigg(    \int_0^{t}  \big(\abs{\vec a(s)}_{\rho,\sigma}^2+\abs{\vec a(s)}_{\rho,\sigma}^4 \big)ds +
 		  \int_0^{t}  \frac{  \abs{\vec a(s)}_{ \tilde\rho,\sigma}^2}{\tilde\rho-\rho}\,ds\bigg) 
 		  \end{aligned}
 	\end{equation} 
 	holds for  any $k\leq j-1$ with $ i+ k\leq 4$ and   $m+i+k\geq 7$, 
 	 we will show the above estimate  holds for $k=j$.  To do so, applying $\comi z^{\ell+j}\partial_t^i \partial_z^j$, $i+j\leq 4,$ to    equation \eqref{zu}  and  observing  
 \begin{eqnarray*}
 \partial_x^m \xi=(f\partial_x +h\partial_z )\partial_x^m f+\underbrace{ \sum_{1\leq k\leq m}{m\choose k}\Big[(\partial_x^k f)\partial_x^{m-k+1}f+  (\partial_x^k h)\partial_x^{m-k}\partial_zf\Big]}_{\stackrel{\rm def}{=} H_m},
 \end{eqnarray*}
 we obtain, with $H_m$ defined above,   
  \begin{eqnarray*}
  \begin{aligned}
 &\big(\partial_t+u\partial_x +w\partial_z-\nu\partial_z^2\big)  \comi z^{\ell+j} \partial_t^i \partial_x^m  \partial_z^j u=(f\partial_x+h\partial_z) \comi z^{\ell+j} \partial_t^i \partial_z^j \partial_x^m f + \comi z^{\ell+j} \partial_t^i\partial_z^j H_{m}\\
&\quad  +\big[ \comi z^{\ell+j}\partial_t^i\partial_z^j, \ f\partial_x +h\partial_z\big]  \partial_x^m f 
   -  \comi z^{\ell+j} \partial_t^i\partial_z^j\big[ (\partial_x^mw)\partial_z u \big]\\ 
&\quad + \comi z^{\ell+j}\partial_t^i \partial_z^j F_{m}+\big[u\partial_x +w\partial_z-\nu\partial_z^2,  \ \comi z^{\ell+j}\partial_t^i\partial_z^j \big]  \partial_x^m u,
 \end{aligned}
\end{eqnarray*}
 where
  $F_m$ is defined in \eqref{zu} and $[T_1,T_2]=T_1T_2-T_2T_1$ stands for the commutator of two operators $T_1, T_2$.  
  Following the argument used in the proof of Lemma \ref{lefx} (see also the proof of \cite[Lemma 4.3]{lmy}), we have
    \begin{multline*} 
\Big(- \comi z^{\ell+j} \partial_t^i\partial_z^j\big[ (\partial_x^mw)\partial_z u \big]+\comi z^{\ell+j} \partial_t^i\partial_z^j F_{m},\ \comi z^{\ell+j}  \partial_t^i\partial_x^m  \partial_z^j u\Big)_{L^2}\\
 +\Big(  \big[u\partial_x+ w\partial_z-\nu\partial_z^2,  \ \comi z^{\ell+j}\partial_t^i\partial_z^j \big]  \partial_x^m u,\ \comi z^{\ell+j}  \partial_t^i\partial_x^m  \partial_z^j u\Big)_{L^2}\\
 	\leq \frac{\nu}{4}\big\|\partial_z\big(\comi z^{\ell+j}  \partial_t^i\partial_x^m  \partial_z^j u\big)\big\|_{L^2}^2+ C\frac{   [\inner{m+i+j-7}!]^{2\sigma}}{\rho^{2(m+i+j-7)}}   \big(\abs{\vec a}_{ \rho,\sigma}^2+   \abs{\vec a}_{ \rho,\sigma}^4 \big),
 \end{multline*}  
 and
 \begin{multline*}
 \Big( \comi z^{\ell+j} \partial_t^i\partial_z^j H_{m},\ \comi z^{\ell+j}  \partial_t^i\partial_x^m  \partial_z^j u\Big)_{L^2}
 +\Big(  \big[f\partial_x+ h\partial_z,  \ \comi z^{\ell+j}\partial_t^i\partial_z^j \big]  \partial_x^m f,\ \comi z^{\ell+j}  \partial_t^i\partial_x^m  \partial_z^j u\Big)_{L^2}\\
 	\leq \frac{\nu}{4}\big\|\partial_z\big(\comi z^{\ell+j}  \partial_t^i\partial_x^m  \partial_z^j u\big)\big\|_{L^2}^2+ C\frac{   [\inner{m+i+j-7}!]^{2\sigma}}{\rho^{2(m+i+j-7)}}   \big(\abs{\vec a}_{ \rho,\sigma}^2+   \abs{\vec a}_{ \rho,\sigma}^4 \big).	
 \end{multline*}
  Hence, combining the above inequalities gives
 \begin{equation} \label{elb}
 \begin{aligned}
 	&\frac{1}{2}\frac{d}{dt}\norm{\comi z^{\ell+j} \partial_t^i \partial_x^m  \partial_z^j u}_{L^2}^2+\frac{\nu}{2}\big\|\partial_z\big(\comi z^{\ell+j} \partial_t^i \partial_x^m  \partial_z^j u\big)\big\|_{L^2}^2 \\
& \leq  \Big|\nu \int_{\mathbb R} \big( \partial_t^i \partial_x^m  \partial_z^j u\big)   \big(\partial_t^i\partial_x^m  \partial_z^{j+1} u\big)\big|_{z=0}  dx \Big| +C\frac{   [\inner{m+i+j-7}!]^{2\sigma}}{\rho^{2(m+i+j-7)}}   (\abs{\vec a}_{ \rho,\sigma}^2+   \abs{\vec a}_{ \rho,\sigma}^4)\\
&\quad+\Big( (f\partial_x+h\partial_z) \comi z^{\ell+j} \partial_t^i \partial_x^m\partial_z^j  f ,\ \comi z^{\ell+j} \partial_t^i \partial_x^m  \partial_z^j u\Big)_{L^2}.
\end{aligned}
 \end{equation}
For the boundary term, since  $\norm{\omega}_{L^\infty(\mathbb R_+)}^2\leq 2\norm{\partial_z\omega}_{L^2(\mathbb R_+)} \norm{ \omega}_{L^2(\mathbb R_+)} $ if $\omega\rightarrow 0$ as $z\rightarrow+\infty$ and 
\begin{eqnarray*}
	\nu\partial_t^i\partial_x^m  \partial_z^{j+1} u|_{z=0} 
	&=&\partial_t^{i+1}\partial_x^m  \partial_z^{j-1} u|_{z=0} 
	+\partial_t^{i}\partial_x^m  \partial_z^{j-1} \big(u\partial_x u+w\partial_zu-f\partial_xf-h\partial_zf\big)\big|_{z=0}\\
	&=&\partial_t^{i+1}\partial_x^m  \partial_z^{j-1} u|_{z=0} 
	+\partial_t^{i}\partial_x^m  \partial_z^{j-1} \big(u\partial_x u -f\partial_xf\big)\big|_{z=0}\\
	&&+  \sum_{k=1}^{j-1}{{j-1}\choose k}\partial_t^{i}\partial_x^m\Big(  ( \partial_z^{k} w)\partial_z^{j-k} u-( \partial_z^{k}h)\partial_z^{j-k}f\Big)\Big|_{z=0},
\end{eqnarray*}
we have
\begin{eqnarray*}
\begin{aligned}
	& \nu^2\frac{\rho}{m^{\sigma}}\norm{    \big(\partial_t^i\partial_x^m  \partial_z^{j+1} u\big) |_{z=0} }_{L_x^2}^2
\leq  C \frac{\rho }{m^{\sigma}}\norm{  \partial_t^{i+1}\partial_x^m  \partial_z^{j-1} u }_{L^2} \norm{  \partial_t^{i+1}\partial_x^m  \partial_z^{j} u }_{L^2} \\
& \quad  +C \frac{\rho }{m^{\sigma}}\norm{  \partial_t^{i}\partial_x^m  \partial_z^{j-1} \big(u\partial_x u -f\partial_xf\big) }_{L^2} \norm{ \partial_t^{i}\partial_x^m  \partial_z^{j} \big(u\partial_x u -f\partial_xf\big)}_{L^2} \\
& \quad  +C\sum_{k=1}^{j-1}  \frac{\rho }{m^{\sigma}}\norm{ \partial_t^{i}\partial_x^m\big( ( \partial_z^{k} w)\partial_z^{j-k} u-( \partial_z^{k}h)\partial_z^{j-k}f\big) }_{L^2} \times \norm{\partial_z \partial_t^{i}\partial_x^m\big( ( \partial_z^{k} w)\partial_z^{j-k} u-( \partial_z^{k}h)\partial_z^{j-k}f\big)}_{L^2} \\
&\leq  \norm{  \partial_t^{i+1}\partial_x^m  \partial_z^{j} u }_{L^2}^2+C\frac{   [\inner{m+i+j-7}!]^{2\sigma}}{\rho^{2(m+i+j-7)}}     \inner{\abs{\vec a}_{ \rho,\sigma}^2+\abs{\vec a}_{ \rho,\sigma}^4},
\end{aligned}
\end{eqnarray*}
where we have used again the argument similar to  the proof of Lemma \ref{lefx}, and the estimate for $p+q\leq 4$,
\begin{eqnarray*}
	 \norm{\comi z^{\ell+q}\partial_t^p\partial_x^m\partial_z^qu}_{L^2}+\norm{\comi z^{\ell+q}\partial_t^p\partial_x^m\partial_z^q f}_{L^2}
	 \leq 
	  \frac{   [\inner{m+p+q-7}!]^{ \sigma}}{\rho^{ (m+p+q-7)}}\abs{\vec a}_{\rho,\sigma}.
\end{eqnarray*}
Moreover,
\begin{eqnarray*}
	    \frac{m^{\sigma}}{\rho}\norm{ \big( \partial_t^i \partial_x^m  \partial_z^j u\big)   |_{z=0} }_{L_x^2}^2  &
	    \leq &C \frac{m^{\sigma}}{\rho }\norm{ \partial_t^i \partial_x^m  \partial_z^j u }_{L^2}\norm{ \partial_t^i \partial_x^m  \partial_z^{j+1} u }_{L^2} \\
 &\leq & \frac{\nu}{4} \norm{ \partial_t^i \partial_x^m  \partial_z^{j+1} u }_{L^2}^2+C\frac{m^{2\sigma}}{\rho^2 }\norm{ \partial_t^i \partial_x^m  \partial_z^j u }_{L^2}^2. 
\end{eqnarray*}
Hence,
\begin{multline*}
 \Big|\nu \int_{\mathbb R} \big( \partial_t^i \partial_x^m  \partial_z^j u\big)   \big(\partial_t^i\partial_x^m  \partial_z^{j+1} u\big)\big|_{z=0}  dx \Big|  
 \leq  \frac{\nu}{4} \norm{ \partial_t^i \partial_x^m  \partial_z^{j+1} u }_{L^2}^2+\norm{  \partial_t^{i+1}\partial_x^m  \partial_z^{j} u }_{L^2}^2\\
 +C\frac{m^{2\sigma}}{\rho^2 }\norm{ \partial_t^i \partial_x^m  \partial_z^j u }_{L^2}^2+C\frac{   [\inner{m+i+j-7}!]^{2\sigma}}{\rho^{2(m+i+j-7)}} \inner{    \abs{\vec a}_{ \rho,\sigma}^2+\abs{\vec a}_{ \rho,\sigma}^4},
\end{multline*}
which together with  \eqref{elb} implies
\begin{eqnarray*}
	 \begin{aligned}
 	&\norm{\comi z^{\ell+j} \partial_t^i \partial_x^m  \partial_z^j u(t)}_{L^2}^2+ \frac{\nu}{2} \int_0^t\norm{ \comi z^{\ell+j} \partial_t^i \partial_x^m  \partial_z^{j+1} u(s) }_{L^2}^2ds \\
& \leq \norm{\comi z^{\ell+j} \partial_t^i \partial_x^m  \partial_z^j u(0)}_{L^2}^2+ C\frac{m^{2\sigma}}{\rho^2 }\int_0^t \norm{ \partial_t^i \partial_x^m  \partial_z^j u(s)}_{L^2}^2ds+\int_0^t \norm{ \partial_t^{i+1} \partial_x^m  \partial_z^j u(s)}_{L^2}^2ds \\
&\quad+C\frac{   [\inner{m+i+j-7}!]^{2\sigma}}{\rho^{2(m+i+j-7)}}  \int_0^t (\abs{\vec a(s)}_{ \rho,\sigma}^2+   \abs{\vec a(s)}_{ \rho,\sigma}^4)ds+\Big( (f\partial_x+h\partial_z) \comi z^{\ell+j} \partial_t^i \partial_x^m\partial_z^j  f ,\ \comi z^{\ell+j} \partial_t^i \partial_x^m  \partial_z^j u\Big)_{L^2} \\
& \leq C\frac{  [(m+i+j-7)!]^{2\sigma}  }{\rho^{2(m+i+j-7)}}  \inner{\norm{(u_0,f_0)}_{2\rho_0,\sigma,8}^2+\norm{(u_0,f_0)}_{2\rho_0,\sigma,8}^4}\\
&\quad  +e^{CC_*^2}\frac{  [(m+i+j-7)!]^{2\sigma}  }{\rho^{2(m+i+j-7)}}    \bigg(    \int_0^{t}  \big(\abs{\vec a(s)}_{\rho,\sigma}^2+\abs{\vec a(s)}_{\rho,\sigma}^4 \big)ds +
 		  \int_0^{t}  \frac{  \abs{\vec a(s)}_{ \tilde\rho,\sigma}^2}{\tilde\rho-\rho}\,ds\bigg) \\
&\quad+\Big( (f\partial_x+h\partial_z) \comi z^{\ell+j} \partial_t^i \partial_x^m  \partial_z^j f ,\ \comi z^{\ell+j} \partial_t^i \partial_x^m  \partial_z^j u\Big)_{L^2},
\end{aligned}
\end{eqnarray*}
where in the last inequality we have used  the induction assumption \eqref{ina}. Similarly,
\begin{eqnarray*}
	 \begin{aligned}
 	&\norm{\comi z^{\ell+j} \partial_t^i \partial_x^m  \partial_z^j f(t)}_{L^2}^2+ \frac{\mu}{2}\int_0^t\norm{ \comi z^{\ell+j} \partial_t^i \partial_x^m  \partial_z^{j+1} f(s)}_{L^2}^2ds \\
& \leq C\frac{  [(m+i+j-7)!]^{2\sigma}  }{\rho^{2(m+i+j-7)}}  \inner{\norm{(u_0,f_0)}_{2\rho_0,\sigma,8}^2+\norm{(u_0,f_0)}_{2\rho_0,\sigma,8}^4}\\
&\quad   +e^{CC_*^2}\frac{  [(m+i+j-7)!]^{2\sigma}  }{\rho^{2(m+i+j-7)}}    \bigg(    \int_0^{t}  \big(\abs{\vec a(s)}_{\rho,\sigma}^2+\abs{\vec a(s)}_{\rho,\sigma}^4 \big)ds +
 		  \int_0^{t}  \frac{  \abs{\vec a(s)}_{ \tilde\rho,\sigma}^2}{\tilde\rho-\rho}\,ds\bigg)\\
&\quad+\Big( (f\partial_x+h\partial_z) \comi z^{\ell+j} \partial_t^i \partial_x^m\partial_z^j  u ,\ \comi z^{\ell+j} \partial_t^i \partial_x^m  \partial_z^j f\Big)_{L^2}.
\end{aligned}
\end{eqnarray*}
Taking the summation of the above two estimates and noticing
\begin{multline*}
	\Big( (f\partial_x+h\partial_z) \comi z^{\ell+j} \partial_t^i \partial_x^m  \partial_z^j f ,\ \comi z^{\ell+j} \partial_t^i \partial_x^m  \partial_z^j u\Big)_{L^2}
	+\Big( (f\partial_x+h\partial_z) \comi z^{\ell+j} \partial_t^i  \partial_x^m \partial_z^j u ,\ \comi z^{\ell+j} \partial_t^i \partial_x^m  \partial_z^j f\Big)_{L^2}=0,
\end{multline*}
we show \eqref{ina} holds for $k=j$.   Thus  the proof of the proposition is  completed. 
   \end{proof}

 \subsection{2D MHD boundary layer}\label{subpr}
 
 By combining  the estimates in Propositions \ref{prpu}-\ref{prp+3} and \ref{prla}-\ref{lemnor},  we obtain the a priori estimate \eqref{aes} in Theorem \ref{apriori} that  enables us to prove the well-posedness of the MHD boundary layer system \eqref{mhd+}.  Precisely, for given intial datum $(u_0,f_0)\in X_{2\rho_0,\sigma,8}$, 
 as  in \cite[Section 7]{MR4055987},  we first construct 
 the  approximate solutions $(u_\eps,f_\eps)\in L^\infty\big([0,   T_\eps]; X_{3\rho_0/2, \sigma, 8}\big)$ to the regularized MHD boundary layer system
\begin{equation}
\label{repradtl}
\left\{
	\begin{aligned}
		&\big(\partial_t  +u_\eps\partial_x +w_\eps\partial_z -\eps\partial_x^2-\nu\partial_z^2\big) u_\eps=\xi_\eps,\\
		&\big(\partial_t  +u_\eps\partial_x +w_\eps\partial_z   -\eps\partial_x^2-\mu\partial_z^2\big) f_\eps= \eta_\eps,\\
		&\big(\partial_t  +u_\eps\partial_x +w_\eps\partial_z   -\eps\partial_x^2-\mu\partial_z^2\big) h_\eps= f_\eps\partial_xw_\eps-h_\eps\partial_xu_\eps,\\
		&\partial_xu_\eps+\partial_zw_\eps=\partial_xf_\eps+\partial_zh_\eps=0,\\
		& (u_\eps,w_\eps)|_{z=0}=(\partial_zf_\eps,h_\eps)|_{z=0}=(0,0), \quad (u_\eps, f_\eps)|_{z\rightarrow+\infty}=(0,0),\\
		& (u_\eps,f_\eps)|_{t=0}=(u_0, f_0),
	\end{aligned}
	\right.
\end{equation}
with $\xi_\eps=(f_\eps\partial_x+h_\eps\partial_z)f_\eps$ and $\eta_\eps=(f_\eps\partial_x+h_\eps\partial_z)u_\eps$. Then we derive a uniform estimate on the approximate solutions $(u_\eps, f_\eps)$  so that we can
take the $\eps\rightarrow 0$ to have existence of solution in a time interval independent of $\eps$.           
For this,  we  define $\vec a_\eps$  in the same way as   $\vec a$  given in Definition \ref{deoa},  with the functions  replaced accordingly by those derived from \eqref{repradtl}.   Then the a priori estimate \eqref{aes} in Theorem \ref{apriori}  holds with $\vec a$ replaced by $\vec a_\eps$.  Hence, we can derive,  repeating the argument in \cite[Section 6]{lmy}, the uniform upper bound with respect to $\eps$ of  the  approximate solutions  $(u_\eps,f_\eps)$ in $L^\infty\inner{[0,T]; X_{\rho,\sigma,4}}$ for some $0<\rho<\rho_0$ and some $T$ independent of $\eps$. 
By taking $\varepsilon\rightarrow 0$,  we conclude, by compactness argument that  the limit $(u,f)$ solves \eqref{mhd+}. The uniqueness of solution follows from a similar argument used in   \cite[Subsection 8.2]{MR4055987}.  Therefore, 
We complete the proof of  Theorem \ref{th2d} for $\sigma=3/2$.
We remark   that as in  \cite[Section 8]{2017ly},   it is straightforward to modify the proof for $1<\sigma<3/2$.

 \section{3D MHD boundary layer} \label{sec3d}
Now we consider the 3D MHD boundary layer and use $(u,v,w)$ and $(f,g,h)$ to denote velocity and magnetic fields respectively, and denote by $(x,y,z)$ the spatial variables in $\mathbb R_+^3$.  Then the
  MHD boundary layer system \eqref{3dimmhd} in three-dimensional space is
 \begin{equation}\label{fla}
	\left\{
	\begin{aligned}
		& \big(\partial_t+u\partial_x +v\partial_y +w\partial_z-\nu\partial_z^2\big)u-\big(f\partial_x+g\partial_y+h\partial_z)f =0,\\
		& \big(\partial_t+u\partial_x +v\partial_y +w\partial_z-\nu\partial_z^2\big)v-\big(f\partial_x+g\partial_y+h\partial_z)g =0,\\
		& \big(\partial_t+u\partial_x +v\partial_y +w\partial_z-\mu\partial_z^2\big)f-\big(f\partial_x+g\partial_y+h\partial_z)u=0,  \\
		& \big(\partial_t+u\partial_x +v\partial_y +w\partial_z-\mu\partial_z^2\big)g-\big(f\partial_x+g\partial_y+h\partial_z)v=0,  \\ 
		& \big(\partial_t+u\partial_x +v\partial_y +w\partial_z-\mu\partial_z^2\big) h=f\partial_xw+g\partial_yw-h\partial_xu-h\partial_y v,  		 
	\end{aligned}
	\right.
\end{equation}
 with the divergence free and initial-boundary conditions
\begin{eqnarray*}
\left\{
\begin{aligned}
&	\partial_xu+\partial_yv+\partial_zw=\partial_xf+\partial_yg+\partial_zh=0,\\
&(u,v,w)|_{z=0}=(\partial_zf,\partial_zg,h)_{z=0}=(0,0,0), \qquad (u,v,f,g)|_{z\rightarrow +\infty}=(0,0,0,0),\\
&(u,v)|_{t=0}=(u_0,v_0),\quad (f,g)|_{t=0}=(f_0,g_0).
\end{aligned}
\right.
\end{eqnarray*}
The proof of well-posedness of this system in Gevrey function space with index $3/2$
 is  similar to the proof for  2D case with slight modification.   Precisely,  instead of the scalar auxiliary functions $\mathcal U, \lambda$ and $\delta$ in \eqref{mau} and \eqref{lamd}  we introduce here  the vector-valued functions
  $\bm{\mathcal U}=(\mathcal U_1, \mathcal U_2)$,  $\bm{\lambda}=(\lambda_1,\lambda_2, \tilde \lambda_1, \tilde\lambda_2)$ and $\bm{\delta}=(\delta_1, \delta_2, \tilde \delta_1, \tilde\delta_2)$, where   
\begin{eqnarray*}
\left\{
 		\begin{aligned}
&   \big(\partial_t+u\partial_x+v\partial_y+w\partial_z-\nu\partial_z^2\big)    \int_0^z\mathcal U_1(t,x,y,\tilde z) d\tilde z  =  -\partial_x w(t,x,y,z),\\
& \big(\partial_t+u\partial_x+v\partial_y+w\partial_z-\nu\partial_z^2\big)    \int_0^z \mathcal U_2(t,x,y,\tilde z) d\tilde z  =  -\partial_y w(t,x,y,z),\\
& \mathcal U_j|_{t=0}=0, \quad \partial_z\mathcal U_j|_{z=0}=\mathcal U_j|_{z\rightarrow+\infty}=0,\quad j=1,2,
   \end{aligned}
   \right.
\end{eqnarray*}
and 
\begin{eqnarray*}
\left\{
\begin{aligned}
	&\lambda_1= \partial_x u-(\partial_z u)\int_0^z\mathcal U_1  d\tilde z,\qquad   \lambda_2= \partial_y u-(\partial_z u)\int_0^z \mathcal U_2  d\tilde z,\\
	&\tilde\lambda_1= \partial_x v-(\partial_z v)\int_0^z\mathcal U_1  d\tilde z,\qquad   \tilde\lambda_2= \partial_y v-(\partial_z v)\int_0^z \mathcal U_2  d\tilde z,\\
	&\delta_1= \partial_x f-(\partial_z f)\int_0^z\mathcal U_1  d\tilde z,\qquad   \delta_2= \partial_y f-(\partial_z f)\int_0^z \mathcal U_2  d\tilde z,\\
	&\tilde\delta_1= \partial_x g-(\partial_z g)\int_0^z\mathcal U_1  d\tilde z,\qquad   \tilde\delta_2= \partial_y g-(\partial_z g)\int_0^z \mathcal U_2  d\tilde z.
	\end{aligned}
	\right.
\end{eqnarray*}
Moreover,  corresponding to \eqref{varpsi}, set $\bm \xi=(\xi_1,\xi_2)$ and $\bm \eta=(\eta_1,\eta_2)$ by 
\begin{equation*} 
\left\{
\begin{aligned}
		&\xi_1=(f\partial_x+g\partial_y+h\partial_z)f, \quad \xi_2=(f\partial_x+g\partial_y+h\partial_z)g,\\
		&\eta_1=(f\partial_x+g\partial_y+h\partial_z)u, \quad \eta_2=(f\partial_x+g\partial_y+h\partial_z)v.
	\end{aligned}
	\right.
\end{equation*}
We remark that  as in the 2D case,  here we can also apply the cancellation mechanism so that the highest order term $\partial_x w$ does not appear in the evolution equations solved by $\xi_j,\eta_j, j=1,2$.    
With the above functions,  we  set accordingly   $\vec a=(u,v,f,g, \bm{\mathcal U}, \bm \lambda,\bm\delta,\bm\xi,\bm\eta)$  and define $\abs{\vec a}_{\rho,\sigma}$ in the same way as in Definition \ref{deoa} with the tangential derivative  $\partial_x^m$  replaced by 
  $\partial_x^{\alpha_1}\partial_y^{\alpha_2}$.
    Then the a priori estimate stated in Theorem \ref{apriori} also holds for the new function
    $\vec a$. For example,  we can repeat the argument used  in Subsections \ref{subu} to get the desired estimate on the $L^2$ norm of  $\partial_x^m \bm{\mathcal U}$ and $\partial_y^m\bm{\mathcal U}$, and the estimate for  $\partial_x^{\alpha_1}\partial_y^{\alpha_2}\bm{ \mathcal U}$ will follow from the inequality
\begin{eqnarray*}
	\norm{\partial_x^{\alpha_1}\partial_y^{\alpha_2} \bm{ \mathcal U}}_{L^2(\mathbb R_+^3)}\leq C \inner{\norm{\partial_x^{\alpha_1+\alpha_2}  \bm{ \mathcal U}}_{L^2(\mathbb R_+^3)}+\norm{\partial_y^{\alpha_1+\alpha_2} \bm{ \mathcal U}}_{L^2(\mathbb R_+^3)}}.
\end{eqnarray*}
Similar argument holds for  the estimates on the other  functions $\bm\lambda, \bm\delta, \cdots$.
  Again from the  a priori estimate we can derive the existence and uniqueness of solution to the 3D MHD boundary layer system \eqref{fla} with corresponding initial and boundary conditions.  
  Since there is no extra difficulty in the proof for the 3D case, we omit the detail  for brevity.

 \bigskip
 {\bf Acknowledgements}. The research of the first author was supported by NSFC (Nos. 11961160716, 11771342, 11871054) and the  Fundamental Research Funds for the Central Universities(No.2042020kf0210).    
   The research of the second author was supported by the General Research Fund of Hong Kong, CityU No.11304419. 
   

\end{document}